\documentclass[10pt]{article}
\usepackage{amsfonts}
\usepackage{amsmath}
\usepackage{amssymb}
\usepackage{amsthm}
\usepackage{mathrsfs}
\usepackage{graphicx}
\usepackage{tabularx}
\usepackage[includemp,body={398pt,550pt},footskip=30pt,%
            marginparwidth=60pt,marginparsep=10pt]{geometry}

\numberwithin{equation}{section}  
\numberwithin{figure}{section} 
\newtheorem{proposition}{P{\scriptsize ROPOSITION}}[section]
\newtheorem{theorem}[proposition]{T{\scriptsize HEOREM}}
\newtheorem{lemma}[proposition]{L{\scriptsize  EMMA}}
\newtheorem{corollary}[proposition]{C{\scriptsize OROLLARY}}
\newtheorem{remark}{R{\scriptsize  EMARK}}[section]
\newtheorem{definition}{D{\scriptsize  EFINITION}}[section]

\newtheorem{example}{E{\scriptsize  xample}}[section]

\newtheorem{algorithm}{A{\scriptsize LGORITHM}}[section]

\title {\Large{\bf  
A New Real Structure-preserving Quaternion QR Algorithm
}}
\author{Zhi-Gang Jia$^{1}$\thanks{Email: zhgjia@jsnu.edu.cn. 
Supported by  National Natural Science Foundation of China  under grant 11201193, TAPP (PPZY2015A013) and  PAPD  of Jiangsu Higher Education Institutions.
},
  Musheng Wei$^{2}$\thanks{ Email: mwei@shnu.edu.cn. Supported by  National Natural Science Foundation of China  under grant 11171289.},
  Mei-Xiang Zhao$^{3,1}$, and Yong Chen$^1$
 \\
 \\
$1.$ School of Mathematics and Statistics, Jiangsu Normal University,\\
Xuzhou 221116, P. R. China\\
$2.$ College of Mathematics and Science, Shanghai Normal University,\\
Shanghai 200234, China.\\
$3.$ School of Information and Electrical Engineering, China University \\
of Mining and Technology,
Xuzhou 221116, China
}

\date{}

\begin{document}
\maketitle

\begin{abstract}
New real structure-preserving decompositions are introduced to develop fast and robust algorithms for the (right) eigenproblem of general quaternion matrices. Under the orthogonally $JRS$-symplectic transformations, 
the Francis $JRS$-QR step and the $JRS$-QR algorithm are firstly proposed for $JRS$-symmetric matrices and then applied to calculate  the   Schur forms of quaternion matrices. A novel quaternion Givens matrix is defined and utilized to compute the QR factorization of quaternion Hessenberg matrices. An implicit double shift quaternion QR algorithm is presented with a technique for automatically choosing shifts and within real operations. Numerical experiments are provided to demonstrate the efficiency and accuracy of newly proposed  algorithms.
\end{abstract}

{\bf Key words.} structured matrices; structure-preserving method; quaternion QR algorithm; quaternion eigenvalue problem.

\section{Introduction}\label{s:intr}
Quaternion matrices play an increasing important role  in many  fields of scientific research, both in theory and applications. The topics of quaternions are viewed of interest if the result is rather different than that of real and complex cases or the method is novel. The convenience of geometric representation and the stability of calculation  make quaternions the favourite of scientists and engineers when they develop mathematical models to simulate and analysis physics phenomena.  

Quaternion was introduced to represent points in space by Sir William Rowan Hamilton on Monday 16 October 1843 in Dublin \cite{hamil67,hank80}. During the remainder of his life, Hamilton tried hard to popularize quaternions by studying and teaching them.  He founded a school of ``quaternionists'', and wrote several books to promote quaternions.  {\it Elements of Quaternions} \cite{hamil69} is his last and longest  book. The team of promoting quaternions expanded  quickly,  not only including  Hamiliton and his students. Finkelstein et al \cite{fjs79,fjss62}  built the  foundations of  quaternionic quantum mechanics; Dixon \cite{dixon94}, G\"{u}rsey and Tze \cite{gutz96}  renewed interest in
algebrization and geometrization of physical theories by
non-commutative fields; and many others.
Primarily due to their utility in describing spatial rotations, quaternions have been widely used in and not limited in computer graphics \cite{shoe85}, bioinformatics \cite{shu10},  control theory and physics  since the late 20th century.

Recently, the  book {\it Topics in Quaternion Linear Algebra} \cite{rodman14}, written by Leiba Rodam,  devotes  entirely quaternionic linear algebra and matrix analysis, consisting of two parts. In the first part, fundamental properties and constructions of quaternionic linear algebra are explained, including matrix decompositions, numerical ranges, Jordan and Kronecker canonical forms, etc. In the second part,  the canonical forms of quaternion pencils with symmetries and the exposition approaches that of a research monograph are emphasised. This book  is an excellent reference source for working mathematicians in both theoretical  and applied areas.

Because of noncommutative multiplication of quaternions,  we have two different quaternionic eigenvalues:  the left eigenvalue  and the right eigenvalue. The right eigenvalue theory of quaternion matrices parallels that of complex eigenvalues of complex matrices in some sense, but the behavior of left eigenvalues is quite unexpected \cite{zhang97} and references therein. Most of practical quaternion models require to calculate the right eigenvalues and corresponding eigenvectors  of  quaternion matrices, while the investigation of left eigenvalues is mainly driven by purely mathematical interest.  The distribution of the left and right eigenvalues of quaternion matrices has been well studied by mathematicians. For instance, Zhang \cite{zhang07} proposed the Ger\v{s}gorin type theorems for  right eigenvalues  and  left eigenvalues. On the contrast,  there  is still no  systematic approach feasible for calculating the left eigenvalues of quaternion matrices with dimensions higher than three, and there is an extreme lack of  fast and stable algorithms of computing the right eigenvalues of general quaternion matrices as well.

Non-commutativity of quaternions blocks lots of classic algorithms being directly used to solve quaternionic (right) eigenproblems.  People have two choices of computing the right eigenvalues of general quaternion matrices:  the  quaternion QR algorithm \cite{bbm89}  and the well-known real or complex counterpart method  \cite{jiang05,leo00, zhang97}. Bunse-Gerstner, Byers and Mehrmann \cite{bbm89}  made a notable contribution  on proposing the double-implicit-shift strategy  and the Francis QR algorithm for quaternion matrices, and on calculating the quaternion Schur form with quaternion unitary similarity transformations. They also proposed the  underlying theory of the quaternion QR algorithm, including the uniqueness and the preservation of the Hessenberg form,  and indicated that such algorithm is backward stable. As the second choice, the  real or complex counterpart method  equivalently transforms the quaternionic right eigenproblem  into the eigenproblem of a real (or complex) matrix with dimension expanded four (or two) times. 
Its efficiency is now challenged by the increasing dimensions of quaternion matrices from applied fields, because of expanding the necessary operation flops and   storage space by several times.  
This new trouble is due to  overlooking algebraic structures of the real (or complex) counterpart.

The real structure-preserving strategy is to develop fast and stable algorithms relying on  structures of 
 the quaternion matrix and its real counterpart and  only processing real operations.  The aim is to combine the stability of quaternion operations and the rapidity of real calculations without dimension expanding. In essence, the real structure-preserving algorithms have comparable operation  flops and storage space with  the algorithms based on quaternion operations. The multiple symmetry structures of the real counterpart 
 were introduced in \cite{jwl13} and had been applied into computing many decompositions of quaternion matrices. The  real structure-preserving tridiagonalization algorithm in \cite{jwl13} reduced  a Hermitian quaternion matrix  into a real symmetric and  tridiagonal matrix of the same order, with the eigen information preserved. A structure-preserving LU decomposition based on the structure-preserving Gauss transformation was proposed for quaternion matrices in \cite{wama13}. Four kinds of quaternion Householder based transformations were compared with each other on their computation amounts and assignment numbers in the calculation of the QRD and SVD of quaternion matrices in \cite{lwzz16}. These real structure-preserving algorithms  have comparable stability and accuracy with the quaternion-operation-based algorithms. 
To the best of our knowledge,  there are still no real structure-preserving algorithms of solving the right eigenvalue problem of non-Hermitian quaternion matrices, which is a very difficult and important problem in quaternionic linear algebra and its applications. We will propose a new real structure-preserving QR algorithm for general quaternion matrices, with costing about a quarter of  arithmetic operations and storage space of applying  the conventional QR algorithm on  their real counterparts.

%%%%%%%%%%%%%%%%%%%%%%%%%%%%%%%%%%%%%%%%%%%%%%%%%%%%%%%%%%%%%%%%%%%%%%%%%%%%%%
This paper is organized as follows.
 In Section \ref{s:pre}, we present some properties of quaternion matrices and the real counterparts. In Section \ref{s:SPD},   we firstly propose  the structure-preserving decompositions, including $JRS$-Hessenberg, QR and Schur decompositions,   and then present the real structure-preserving  $JRS$-Hessenberg QR iteration. In Section \ref{s:QQRA}, we present a new  fast quaternion Francis QR algorithm.
 In Section \ref{s:numex}, we provide four numerical experiments. 
 Finally in Section \ref{s:conclusion} we give several concluding remarks.

%%%%%%%%%%%%%%%%%%%%%%%%%%%%%%%%%%%%%%%%
\section{ Preliminaries}\label{s:pre}
In this section we present some basic results for quaternion matrices and their real counterparts. Let $ \mathbb{H}$ denote the division ring generated by $1,~i,~j$ and $k$, with identity $1$ and 
$$i^2 = j^2 = k^2 = ijk = -1.$$

\subsection{Quaternion matrices and $JRS$-symmetric matrices}\label{ss:qJRS}
A quaternion matrix  $Q \in \mathbb{H}^{m\times n}$ is of the form
$$Q=Q_0+Q_1i+Q_2j+Q_3k,~Q_0,\cdots,Q_3\in\mathbb{R}^{m\times n},$$
and its conjugate transpose is defined as $Q^*=Q_0^T-Q_1^Ti-Q_2^Tj-Q_3^Tk$. A quaternion matrix $Q$ has right linearly independent columns (or in other words, $Q$ is full of column rank)  if and only if $Qx=0$ has a unique solution $x=0$, and moreover, the columns of $Q$ are orthogonal to each other if   $Q^*Q= I$.
The  real counterpart of a quaternion matrix $Q$ is defined in \cite{jwl13} as 
\begin{equation}\label{e:realq}
   \Upsilon_Q\equiv\left[
              \begin{array}{rrrr}
                Q_0 & Q_2& Q_1 & Q_3 \\
                -Q_2 & Q_0 & Q_3 & -Q_1 \\
                -Q_1 & -Q_3 & Q_0 & Q_2 \\
                -Q_3 & Q_1 & -Q_2 & Q_0 \\
              \end{array}
            \right]\in\mathbb{R}^{4m\times 4n}.
\end{equation}
Many computational problems of quaternion matrices can be proceeded by corresponding  real counterparts, with giving a rise  of the dimension-expanding obstacle when the original quaternion matrix is huge.
Such trouble can be solved if we sufficiently apply the structures of real counterparts in the processing of calculation. So  we need to  generalize the definitions of $JRS$-symmetric and symplectic (square) matrices  in  \cite{jwl13} into rectangular matrices.  
\begin{definition}\label{d:JRSsym}
  Define three unitary matrices
\begin{equation*}\label{e:JRS}
                              \ J_n=\left[
                                \begin{matrix}%{cccc}
                                 0 & 0 & -I_n & 0 \\
                                  0 & 0 & 0&-I_n  \\
                                  I_n& 0 & 0 & 0 \\
                                 0 &I_n & 0&  0\\
                                \end{matrix}
                              \right],
 R_n=\left[
                                \begin{matrix}%{cccc}
                                 0&-I_n & 0 & 0 \\
                                  I_n & 0 & 0& 0 \\
                                  0 & 0 & 0 & I_n \\
                                 0 & 0& -I_n& 0\\
                                \end{matrix}
                              \right],
                              \ S_n=\left[
                                \begin{matrix}%{cccc}
                                 0& 0& 0 & -I_n \\
                                 0 & 0 &I_n  & 0 \\
                                  0 & -I_n & 0 &0 \\
                                 I_n& 0&0& 0\\
                                \end{matrix}
                              \right].\end{equation*}
 \begin{itemize}
 \item[$(1)$] A real matrix $M\in\mathbb{R}^{4m\times 4n}$ is called \emph{$JRS$-symmetric} if $J_mMJ_n^T=M$, $R_mMR_n^T=M$ and $S_mMS_n^T=M$.
\item[$(2)$] If $m\le n$,  a matrix  $O\in\mathbb{R}^{4m\times 4n}$ is called   \emph{$JRS$-symplectic} if   $OJ_nO^T=J_m$, $OR_nO^T=R_m$ and $OS_nO^T=S_m$.
\item[$(3)$]  A matrix  $W\in\mathbb{R}^{4n\times 4n}$ is called   \emph{orthogonally $JRS$-symplectic} if it is orthogonal and $JRS$-symplectic.
\end{itemize}
\end{definition}
\noindent
We can see that an $n$-by-$n$ quaternion matrix $Q$ is unitary if and only if its real counterpart $\Upsilon_Q$ is orthogonal; and $\Upsilon_Q$ is orthogonal  if and only if  it is orthogonally $JRS$-symplectic, because $\Upsilon_Q$ is surely $JRS$-symmetric.  

Notice that the set of $JRS$-symmetric matrices is closed under addition and multiplication. 
\begin{lemma}\label{l:JRSprop}
Suppose that $M\in\mathbb{R}^{4m\times 4n}$, $B\in\mathbb{R}^{4m\times 4\ell}$ and $C\in\mathbb{R}^{4\ell\times 4n}$   are $JRS$-symmetric. 
\begin{itemize}
\item[$(1)$] $M$ has a partitioning as
\begin{equation}\label{e:M}
M=\left[
                  \begin{array}{rrrr}
                 M_0 &M_2& M_1 & M_3 \\
                -M_2 & M_0 & M_3 & -M_1 \\
                -M_1 & -M_3 & M_0 & M_2 \\
                -M_3 & M_1 & -M_2 & M_0 \\
                  \end{array}
                \right].
\end{equation}
\item[$(2)$]
For any  $\alpha,\beta\in\mathbb{R}$, $\alpha M+\beta BC$ is $JRS$-symmetric.
\item[$(3)$]Moreover,   if $B$ and $C$  are  $JRS$-symplectic, then $BC$ is also $JRS$-symplectic.
\end{itemize}
\end{lemma}
\begin{proof} We only prove the item $(3)$, because  items $(1)$ and $(2)$ can be proved by  direct calculation. 
Since $B$ and $C$ are $JRS$-symplectic, we have 
$$BJ_\ell B^T=J_m, ~BR_\ell B^T=R_m, ~BS_\ell B^T=S_m,$$
and 
$$CJ_n C^T=J_\ell, ~CR_n C^T=R_\ell, ~CS_n C^T=S_\ell.$$
Then 
$$(BC)J_n (BC)^T=B(CJ_n C^T)B^T=BJ_\ell B^T=J_m,\ ~ $$
$$(BC)R_n (BC)^T=B(CR_n C^T)B^T=BR_\ell B^T=R_m,$$
$$(BC)S_n (BC)^T=B(CS_n C^T)B^T=BS_\ell B^T=S_m.\ ~$$
According to the second item in Definition \ref{d:JRSsym},  $BC$ is $JRS$-symplectic.
\end{proof}

With the real counterpart as a bridge,  many properties of quaternion matrices can be obtained through studying $JRS$-symmetric matrices. This is based on  an important discovery: 
\begin{theorem}\label{t:JRSrealq}
A matrix $M\in\mathbb{R}^{4m\times 4n}$ is $JRS$-symmetric if and only if $M$ is a real counterpart of a quaternion matrix.
\end{theorem}
\begin{proof}The theorem can be proved by straightforward computation.
\end{proof}

The basic quaternion operations can be proceeded only by real arithmetic based on Lemma \ref{l:JRSprop} and Theorem \ref{t:JRSrealq}.   For instance, suppose that  $Q,~M,~N\in\mathbb{H}^{m\times n},~A\in\mathbb{H}^{m\times \ell},~B\in\mathbb{H}^{\ell \times n}$,  and $\alpha,~\beta\in\mathbb{R}$, then 
\begin{itemize}
\item   $Q=\alpha M+\beta N$ if and only if   $\Upsilon_Q=\alpha\Upsilon_ M+\beta\Upsilon_N$ (\cite{jwl13});
\item   $Q=\alpha AB$ if and only if   $\Upsilon_Q=\alpha\Upsilon_ A\Upsilon_B$ (\cite{jwl13});
\item   $\Upsilon_{\alpha Q^*}=\alpha(\Upsilon_Q)^T$;
\item  $\Upsilon_{ Q^{-1}}=(\Upsilon_Q)^{-1}$ if $Q$ is invertible;
\item  $2\|Q\|_F=\|\Upsilon_Q\|_F$ (\cite{lwzz16}),   $\|Q\|_2=\|\Upsilon_Q\|_2$, and  $\rho(Q)=\rho(\Upsilon_Q)$.
\end{itemize}
\subsection{The quaternion eigenvalue problems}
A pair $(x, \lambda)$ with nonzero vector $x\in\mathbb{H}^n$ and $\lambda\in\mathbb{H}$ is called the right (left) eigenpair of  a quaternion matrix $Q\in\mathbb{H}^{n\times n}$ if
\begin{equation}\label{eigprob}
Qx= x\lambda ~~(Qx= \lambda x).
\end{equation}
The existence of right eigenvalues for any quaternion matrix was first proved by Berenner \cite{brenner51}. The left eigenvalue problem was raised by Cohn \cite{cohn77}  and the existence of left eigenvalues for any quaternion matrix was proved by Wood \cite{wood85} using a topological approach.  
Every  $n$-by-$n$ quaternion matrix has at least one left eigenvalue in $\mathbb{H}$  \cite{wood85}, and however has exactly $n$ right eigenvalues, which are complex numbers with nonnegative imaginary parts \cite{brenner51,lee49}. 
Such right eigenvalues are called standard eigenvalues in \cite{zhang97}.
 Generally, left and right eigenvalues have no strong relation  to each other.  
 But they coincides when $Q$ is a real matrix. 
 Since the right eigenvalues  have been well studied in theory and are more available in many applications,   we  only study the right eigenvalues of quaternion matrices  and use   ``eigenvalue'' to indicate the right eigenvalue for simplicity in the rest of this paper. 

By adopting quaternion scalar products in $\mathbb{H}^n$, we find states in one-to-one correspondence with unit rays of the form
$v = \{x \beta\}$, where $x$ is a normalized vector and $\beta$ is a quaternion phase of unity magnitude.
The
state vector, $x\beta$, corresponding to the same physical state $x$, is an eigenvector with
eigenvalue $\overline{\beta}\lambda \beta$, 
$Q(x\beta)=(x\beta)(\overline{\beta}\lambda \beta).$
For real values of $\lambda$, we find only one eigenvalue, otherwise we can find an infinite eigenvalue spectrum
$[\lambda]=\{\lambda, \overline{\beta}_1\lambda \beta_1, \cdots, \overline{\beta}_\ell\lambda \beta_\ell, \cdots \}$
with $\beta_\ell$ unitary quaternions, called the equivalence class containing $\lambda$.
 The related set of eigenvectors
$\{x, x \beta_1, \cdots , x \beta_\ell, \cdots\}$
represents a ray.   Any two quaternions are similar if and only if their real parts and  modules of imaginary parts are respectively equivalent \cite[Theorem 2.2]{zhang97}. If $\lambda$ is not real then $[\lambda]$ contains only two complex numbers that are a conjugate pair.  
 In fact, if $\lambda=a+bi+cj+dk$ with $c^2+d^2\ne 0$ then we can choose $\beta=\alpha/|\alpha|$ with $\alpha=b+\sqrt{b^2+c^2+d^2}-dj+ck$ such that $\lambda_c=\overline{\beta}\lambda\beta=a+\sqrt{b^2+c^2+d^2}i\in[\lambda].$
For this state the right eigenvalue equation in \eqref{eigprob} becomes
\begin{equation}\label{eigprobcomplex}Qv = v\lambda_c\end{equation}
with $v\in\mathbb{H}^n$ is a representative ray and $\lambda_c\in\mathbb{C}$ is the corresponding standard eigenvalue. We will focus on  computing  the standard eigenvalues of quaternion matrices.

The following important results are recalled from \cite{rodman14} and  \cite{zhang97}:
\begin{theorem}[\cite{rodman14,zhang97}]\label{t:decompositionq}
 Let $Q\in\mathbb{H}^{n\times n}$. Then:
 \begin{itemize}
\item  (Schur's triangularization theorem) there exists a unitary $U\in\mathbb{H}^{n\times n}$ such that $U^*QU$
  is upper triangular with  complex diagonal entries;
\item  if $Q$ is Hermitian, then there exists a unitary $U\in\mathbb{H}^{n\times n}$ such that $U^*QU$ is
diagonal and real;
\item  if $Q$ is skew Hermitian, then there exists a unitary $U\in\mathbb{H}^{n\times n}$ such that $U^*QU$
is diagonal complex matrix with purely imaginary nonzero entries;
\item  if $Q$ is unitary, then there exists a unitary $U\in\mathbb{H}^{n\times n}$ such that $U^*QU$ is
diagonal and consists of unit complex numbers.
\end{itemize}
\end{theorem}
Define the quaternion Jordan block as 
$$\mathscr{J}_m(\lambda)= \left[
              \begin{array}{rrrrr}
                \lambda & 1& 0 & \cdots &0 \\
                0 & \lambda & 1&\cdots & 0 \\
                \vdots & \vdots & \ddots&\ddots  & \vdots \\
                  \vdots &   \vdots&  &  \lambda&1  \\
                0 &  0& \cdots &0 &\lambda \\
              \end{array}
            \right]\in\mathbb{H}^{m\times m}.$$
\begin{theorem}[\cite{rodman14,zhang97}]\label{t:jordan}
 Let $Q\in\mathbb{H}^{n\times n}$. 
 Then there exists an invertible $X\in\mathbb{H}^{n\times n}$ such that 
\begin{equation}\label{e:jordan}X^{-1}QX=\mathscr{J}_{m_1}(\lambda_1)\oplus\cdots \mathscr{J}_{m_p}(\lambda_p),~\lambda_1,\cdots,\lambda_p\in \mathbb{H}.\end{equation}
The form \eqref{e:jordan} is uniquely determined by $Q$ up to arbitrary permutation of diagonal blocks 
%$\{\mathscr{J}_{m_j}(\lambda_j)\}_{j=1}^p$ 
and up to  a replacement of $\lambda_1,\cdots,\lambda_p$ with $\hat{\lambda}_1,\cdots,\hat{\lambda}_p$ within the diagonal blocks 
%$\mathscr{J}_{m_1}(\lambda_1),~\cdots,~\mathscr{J}_{m_p}(\lambda_p)$ 
where  $\hat{\lambda}_s\in [\lambda_s],~s=1,\cdots,p$.
\end{theorem}
The (right) eigenvalues are continuous functions of the quaternion matrix.
\begin{theorem}[\cite{rodman14}]
 Let $Q\in\mathbb{H}^{n\times n}$ and let $\lambda_1,\ldots,\lambda_s$ be all the distinct eigenvalues of $Q$ in the closed upper complex half-plane $\mathbb{C}_+$. Then for every $\epsilon>0$, there exists $\delta>0$ such that if $\Delta Q\in\mathbb{H}^{n\times n}$ satisfies $\|\Delta Q\|<\delta$, then the eigenvalues of $Q+\Delta Q$ are contained in the union 
 $$\bigcup_{t=1}^s\{z\in \mathbb{C}_+:~|z-\lambda_t|<\epsilon\}.$$ 
% Moreover, if $\epsilon$ is sufficiently small, then the
\end{theorem}
\noindent
Recall that the eigenvalues in the closed upper complex half-plane $\mathbb{C}_+$ of a quaternion matrix is called  standard eigenvalues. 

For any two different $n$-dimensional  quaternion vectors $x,y$, if $\|x\|=\|y\|$ and $y^*x\in\mathbb{R}$, there exists a Householder matrix $I-2\omega\omega^*$ with $\omega=(y-x)/\|y-x\|$  maps $y$ to $x$ \cite{bbm89}.  Applying the Householder based transformations, we can calculate the QR factorization of quaternion matrix $A\in\mathbb{H}^{n\times n}$,
 i.e.,  $A=QR$, where $Q\in\mathbb{H}^{n\times n}$ is unitary and $R\in\mathbb{H}^{n\times n}$ is upper triangular \cite{buns86}.  As a milestone work,  Bunse-Gerstner, Byers and Mehrmann in \cite{bbm89}  proposed the practical QR algorithm to calculate the Schur decomposition of a quaternion matrix.  
The bump chasing,  double implicit shift method of the Francis QR iteration \cite{gova13,stewart01} were also carried over to the quaternion case with explicit algorithms listed in the appendix of \cite{bbm89}.
 The quaternion QR algorithm (\cite[Algorithm A5]{bbm89}) is suitable for computing the Schur decomposition of a general quaternion matrix. Unfortunately quaternion arithmetic is quite expensive and to be avoided at all possible. 
 We will show that there is a real equivalent of the Schur form and that the QR algorithm can be adapted to compute it in real arithmetic.

The way to  combine the stability of quaternion operations and the rapidity of real calculations is to  develop  real structure-preserving  algorithms based on the  algebraic symmetry properties  of the real counterpart.  
We find that the decompositions of quaternion matrices can be put into effect by  the $JRS$-symmetry-preserving  transformations of their real counterparts, and meanwhile, the  accompanying dimension-expanding problem caused by the real counterpart method will vanish. This motivates us to develop the structure-preserving Hessenberg reduction and the real Schur form of $JRS$-symmetric matrices at first, and then design a new real structure-preserving Francis QR algorithm for quaternion matrices, which is expected to be fast and strongly backward stable. We emphasize that the real counterpart will not be generated in the newly proposed 
algorithms,  and hence the operations will be directly applied on the real part and three imaginary parts of the quaternion matrix.

\section{The structure-preserving methods}\label{s:SPD}
In this section, we propose the structure-preserving Hessenberg, QR and Schur decompositions of $JRS$-symmetric matrices and  the real structure-preserving $JRS$-QR algorithm.  

Firstly, we recall the fact that orthogonally  $JRS$-symplectic equivalence transformations can preserve  the $JRS$-symmetry \cite{jwl13}. From the second term of  Definition \ref{d:JRSsym}, straightforward calculation indicates that  every orthogonally $JRS$-symplectic matrix $W\in\mathbb{R}^{4n\times 4n}$  has the block structure
\begin{equation}\label{e:w}
W=\left[
              \begin{array}{rrrr}
                W_0 & W_2& W_1 & W_3 \\
                -W_2 & W_0 & W_3 & -W_1 \\
                -W_1 & -W_3 & W_0 & W_2 \\
                -W_3 & W_1 & -W_2 & W_0 \\
              \end{array}
            \right], W_1, \cdots, W_3\in\mathbb{R}^{n\times n}.
\end{equation}
An example of orthogonally $JRS$-symplectic matrix  is the  generalized symplectic Givens rotation defined as
\begin{equation}\label{e:gl}
G_\ell=\left[\begin{smallmatrix}
I_{\ell-1},&           0,& 0,& 0,&           0,& 0,& 0,&           0,& 0,& 0,&            0,& 0\\
0,&  \alpha_0,& 0,& 0,&  \alpha_2,& 0,& 0,&   \alpha_1,& 0,& 0,&   \alpha_3,& 0\\
0,&           0,& I_{n-\ell},& 0,&           0,& 0,& 0,&           0,& 0,& 0,&            0,& 0\\
0,&           0,& 0,& I_{\ell-1},&           0,& 0,& 0,&           0,& 0,& 0,&            0,& 0\\
0,& -\alpha_2,& 0,& 0,&  \alpha_0,& 0,& 0,&  \alpha_3,& 0,& 0,&   -\alpha_1,& 0\\
0,&           0,& 0,& 0,&           0,& I_{n-\ell},& 0,&           0,& 0,& 0,&            0,& 0\\
0,&           0,& 0,& 0,&           0,& 0,& I_{\ell-1},&           0,& 0,& 0,&            0,& 0\\
0,&  -\alpha_1,& 0,& 0,& -\alpha_3,& 0,& 0,&  \alpha_0,& 0,& 0,&   \alpha_2,& 0\\
0,&           0,& 0,& 0,&           0,& 0,& 0,&           0,& I_{n-\ell},& 0,&            0,& 0\\
0,&           0,& 0,& 0,&           0,& 0,& 0,&           0,& 0,& I_{\ell-1},&            0,& 0\\
0,& -\alpha_3,& 0,& 0,&   \alpha_1,& 0,& 0,& -\alpha_2,& 0,& 0,&   \alpha_0,& 0\\
0,&           0,& 0,& 0,&           0,& 0,& 0,&           0,& 0,& 0,&            0,& I_{n-\ell}\\
\end{smallmatrix}\right]
\end{equation}
where $1\le \ell\le n, \alpha_0, \alpha_1,\alpha_2, \alpha_3\in[-1,1]$ and $\alpha_0^2+\alpha_1^2+\alpha_2^2+\alpha_3^2=1$.
Notice that if $\alpha_2\equiv 0$ and $\alpha_3\equiv 0$ then $G_l$  defined by \eqref{e:gl}
 is an $4n\times 4n$ symplectic Givens rotation  $J_s(i,\alpha)$ defined by equation (37) in  \cite{bbmx98}.
Another example is the direct sum of four identical $n$-by-$n$ Householder matrices
\begin{equation*}
(\mathscr{H}_\ell\oplus \mathscr{H}_\ell\oplus \mathscr{H}_\ell\oplus \mathscr{H}_\ell)(v,\beta)
\end{equation*}
where $v$ is a vector of length $n$ with its first $\ell-1$ elements equal to zero
and $\beta$ a scalar satisfying $\beta(\beta v^Tv-2)=0$. 
If $W\in\mathbb{R}^{4n\times 4n}$ is orthogonally  $JRS$-symplectic and   $M\in\mathbb{R}^{4n\times 4n}$ is $JRS$-symmetric then
\begin{eqnarray*}
J_n(W^TM W)J_n^T=(J_nW^T)M(J_nW^T)^T=W^TM W,\ \ \\
R_n(W^TM W)R_n^T=(R_nW^T)M(R_nW^T)^T=W^TM W,\\
S_n(W^TM W)S_n^T=(S_nW^T)M(S_nW^T)^T=W^TM W.\ \
\end{eqnarray*}
This implies that $JRS$-symmetry is preserved by orthogonally $JRS$-symplectic similarity transformations.

\subsection{The upper $JRS$-Hessenberg form}
Now we deduce the upper Hessenberg form of $JRS$-symmetric matrices   under the   orthogonally $JRS$-symplectic transformations.

\begin{definition}\label{d:hess} A $JRS$-symmetric matrix  $H\in\mathbb{R}^{4n\times 4n}$ is called an upper $JRS$-Hessenberg matrix if 
\begin{equation}\label{e:hess}
H=\left[
                  \begin{array}{rrrr}
                 H_0 &H_2& H_1 & H_3 \\
                -H_2 & H_0 & H_3 & -H_1 \\
                -H_1 & -H_3 & H_0 & H_2 \\
                -H_3 & H_1 & -H_2 & H_0 \\
                  \end{array}
                \right],
\end{equation}
where $H_0\in\mathbb{R}^{n\times n}$ is an upper Hessenberg matrix, $H_1, \ H_2,\ H_3\in\mathbb{R}^{n\times n}$ are upper triangular matrices. Moreover if all  subdiagonal elements of $H_0$ are nonzeros, $H$ is called an unreduced upper $JRS$-Hessenberg matrix. 
 \end{definition}

\begin{theorem}\label{t:hess}
Suppose that  a $JRS$-symmetric matrix $M\in\mathbb{R}^{4n\times 4n}$ is of the form \eqref{e:M}.
Then there exists an orthogonally $JRS$-symplectic matrix $W\in\mathbb{R}^{4n\times 4n}$ such that
$WM W^T=H$ is an upper $JRS$-Hessenberg matrix.
\end{theorem}
\begin{proof}
We  prove the assertion  by  induction on the order $n$.
  For $n=1$, it is clear that the theorem is true. Suppose that for the case $1\le n<\ell$,  there exists an orthogonally
  $JRS$-symplectic matrix $\widetilde{W}\in\mathbb{R}^{4n\times 4n}$ such that
  \begin{equation}\label{e:wuwn}
 \widetilde{W}M\widetilde{W}^T=\left[
                  \begin{array}{rrrr}
                 \widetilde{H}_0 & \widetilde{H}_2&  \widetilde{H}_1 &  \widetilde{H}_3 \\
                - \widetilde{H}_2 &  \widetilde{H}_0 & \widetilde{H}_3 & - \widetilde{H}_1 \\
                - \widetilde{H}_1 & - \widetilde{H}_3 &  \widetilde{H}_0 &  \widetilde{H}_2 \\
                - \widetilde{H}_3 &  \widetilde{H}_1 & - \widetilde{H}_2 &  \widetilde{H}_0 \\
                  \end{array}
                \right],
  \end{equation}
where $\widetilde{H}_0\in\mathbb{R}^{n\times n}$ is an upper Hessenberg matrix, $\widetilde{H}_{1,2,3}\in\mathbb{R}^{n\times n}$ are upper triangular matrix.
   For $n=\ell$,
denote
 \begin{equation*}M_0=\left[\begin{smallmatrix}%{cccc}
                                 m^{(0)}_{11} &  m^{(0)}_{12}&  m^{(0)}_{13}  &M^{(0)}_{14} \\
                                 m^{(0)}_{21} &  m^{(0)}_{22} & m^{(0)}_{23} &M^{(0)}_{24}\\
                                  m^{(0)}_{31} & m^{(0)}_{32} & m^{(0)}_{33} &M^{(0)}_{34}\\
                                 M^{(0)}_{41} & M^{(0)}_{42} & M^{(0)}_{43} & M^{(0)}_{44}\\
                             \end{smallmatrix}
                             \right],\ \ \ 
M_1=\left[\begin{smallmatrix}
                                m^{(1)}_{11}  & m^{(1)}_{12}&  m^{(1)}_{13}  &M^{(1)}_{14} \\
                               m^{(1)}_{21} &  m^{(1)}_{22}  &  m^{(1)}_{23} &M^{(1)}_{24}\\
                               m^{(1)}_{31} &  m^{(1)}_{32} &   m^{(1)}_{22}  &M^{(1)}_{34}\\
                               M^{(1)}_{41}& M^{(1)}_{42} & M^{(1)}_{43} &  M^{(1)}_{44}
                               \end{smallmatrix}
                             \right],
 \end{equation*}
 \begin{equation*}M_2=\left[\begin{smallmatrix}
                             m^{(2)}_{11}  &  m^{(2)}_{12}&  m^{(2)}_{13}  &M^{(2)}_{14} \\
                             m^{(2)}_{21} &  m^{(2)}_{22} &m^{(2)}_{23} &M^{(2)}_{24}\\
                             m^{(2)}_{31} &  m^{(2)}_{32} &  m^{(2)}_{33}  &M^{(2)}_{34}\\
                             M^{(2)}_{41} & M^{(2)}_{42} & M^{(2)}_{43} &M^{(2)}_{44}
                                \end{smallmatrix}
                             \right],\ \ \ 
M_3=\left[\begin{smallmatrix}
                                    m^{(3)}_{11}  & m^{(3)}_{12}& m^{(3)}_{13}  &M^{(3)}_{14} \\
                                   m^{(3)}_{21} &   m^{(3)}_{22}  &  m^{(3)}_{23} &M^{(3)}_{24}\\
                                   m^{(3)}_{31} &  m^{(3)}_{32} &  m^{(3)}_{33}  &M^{(3)}_{34}\\
                                   M^{(3)}_{41} & M^{(3)}_{42} & M^{(3)}_{43} & M^{(3)}_{44}
                               \end{smallmatrix}
                             \right],
 \end{equation*}
 in which $m^{(r)}_{st}\in\mathbb{R}$,  $M^{(r)}_{s4}\in\mathbb{R}^{1\times (\ell-3)}$ and $M^{(r)}_{4s}\in\mathbb{R}^{(\ell-3)\times 1}$ 
 and $M^{(r)}_{44}\in\mathbb{R}^{(\ell-3)\times (\ell-3)}$,  $ r=0,\ldots,3$, $s,t=1,2,3$,
  
There are a series  of  generalized symplectic Givens rotations $G_2,G_3,\ldots,G_\ell\in\mathbb{R}^{4n\times 4n}$ such that
$$\widehat{M}:= G_\ell \cdots G_3G_2M (G_\ell\cdots G_3G_2)^T=\left[
              \begin{smallmatrix}
                \widehat{M}_0 & \widehat{M}_2& \widehat{M}_1 & \widehat{M}_3 \\
                -\widehat{M}_2 & \widehat{M}_0 & \widehat{M}_3 & -\widehat{M}_1 \\
                -\widehat{M}_1 & -\widehat{M}_3 & \widehat{M}_0 & \widehat{M}_2 \\
                -\widehat{M}_3 & \widehat{M}_1 & -\widehat{M}_2 & \widehat{M}_0 \\
              \end{smallmatrix}
            \right]$$
with
\begin{equation*}\widehat{M}_0=\left[\begin{smallmatrix}%{cccc}
       m^{(0)}_{11} &  \widehat{m}^{(0)}_{12}&  \widehat{m}^{(0)}_{13} &\widehat{M}^{(0)}_{14} \\
        \gamma_{21} &  \widehat{m}^{(0)}_{22} &  \widehat{m}^{(0)}_{23} &\widehat{M}^{(0)}_{24}\\
        \gamma_{31} & \widehat{ m}^{(0)}_{32} &  \widehat{m}^{(0)}_{33} &\widehat{M}^{(0)}_{34}\\
     \Gamma_{41} & \widehat{M}^{(0)}_{42} & \widehat{M}^{(0)}_{43} & \widehat{M}^{(0)}_{44}\\
                             \end{smallmatrix}
                             \right],\ \ \ 
\widehat{M}_1=\left[\begin{smallmatrix}
              m^{(1)}_{11} &  \widehat{m}^{(1)}_{12}  &  \widehat{m}^{(1)}_{13} &\widehat{M}^{(1)}_{14}\\
                                  0 &  \widehat{m}^{(1)}_{22}  &  \widehat{m}^{(1)}_{23} &\widehat{M}^{(1)}_{24}\\
                              0 &  \widehat{m}^{(1)}_{32} &  \widehat{m}^{(1)}_{33}  &\widehat{M}^{(1)}_{34}\\
                  0 & \widehat{M}^{(1)}_{42} & \widehat{M}^{(1)}_{43} &  \widehat{M}^{(1)}_{44}
                               \end{smallmatrix}
                             \right],
 \end{equation*}
 \begin{equation*}\widehat{M}_2=\left[\begin{smallmatrix}
              m^{(2)}_{11} &  \widehat{m}^{(2)}_{12}  &  \widehat{m}^{(2)}_{13} &\widehat{M}^{(2)}_{14}\\
                   0&  \widehat{m}^{(2)}_{22}  &  \widehat{m}^{(2)}_{23} &\widehat{M}^{(2)}_{24}\\
                       0&  \widehat{m}^{(2)}_{32} &  \widehat{m}^{(2)}_{33}  &\widehat{M}^{(2)}_{34}\\
                         0 & \widehat{M}^{(2)}_{42} & \widehat{M}^{(2)}_{43}&\widehat{M}^{(2)}_{44}
                                \end{smallmatrix}
                             \right],\ \ \ 
\widehat{M}_3=\left[\begin{smallmatrix}
      m^{(3)}_{11} & \widehat{m}^{(3)}_{12}  &  \widehat{m}^{(3)}_{13} &\widehat{M}^{(3)}_{14}\\
                          0 &  \widehat{m}^{(3)}_{22}  &  \widehat{m}^{(3)}_{23} &\widehat{M}^{(3)}_{24}\\
                          0 &  \widehat{m}^{(3)}_{32} &  \widehat{m}^{(3)}_{33}  &\widehat{M}^{(3)}_{34}\\
                          0& \widehat{M}^{(3)}_{42} & \widehat{M}^{(3)}_{43} & \widehat{ M}^{(3)}_{44}
                               \end{smallmatrix}
                             \right],
 \end{equation*}
 where    $\Gamma_{41}=[\gamma_{41},\cdots,\gamma_{\ell 1}]^T$, 
  $\gamma_{s1}=\sqrt{ (m^{(0)}_{s1})^2 +  (m^{(2)}_{s1})^2 + (m^{(1)}_{s1})^2  +(m^{(3)}_{s1})^2}$ ($s=2,3,\ldots, \ell$).
Then we can generate a Householder matrix $\mathscr{H}_2\in\mathbb{R}^{\ell\times\ell}$ such that
$$\mathscr{H}_2\widehat{M}_0(:,1)=[m_{11}^{(0)},\widetilde{\gamma}_{21},0,\cdots,0]^T,$$
 and  process  the orthogonally $JRS$-symplectic transformation
$$\widetilde{M}:= [\mathscr{H}_2\oplus \mathscr{H}_2\oplus \mathscr{H}_2\oplus \mathscr{H}_2]\widehat{M}[\mathscr{H}_2\oplus \mathscr{H}_2\oplus \mathscr{H}_2\oplus \mathscr{H}_2]^T=\left[
              \begin{smallmatrix}
                \widetilde{M}_0 & \widetilde{M}_2& \widetilde{M}_1 & \widetilde{M}_3 \\
                -\widetilde{M}_2 & \widetilde{M}_0 & \widetilde{M}_3 & -\widetilde{M}_1 \\
                -\widetilde{M}_1 & -\widetilde{M}_3 & \widetilde{M}_0 & \widetilde{M}_2 \\
                -\widetilde{M}_3 & \widetilde{M}_1 & -\widetilde{M}_2 & \widetilde{M}_0 
              \end{smallmatrix}
            \right],$$
where  $\widetilde{M}_s=H_2  \widehat{M}_sH_2^T$, $s=0,\ldots,3$.
Note that the submatrix of $\widetilde{M}$ by deleting the $1,\ell+1,2\ell+1,3\ell+1$ rows and columns is a $4(\ell-1)\times 4(\ell-1)$ $JRS$-symmetric matrix.
By the introduction assumption,  the theorem can be proved.
\end{proof}
\begin{corollary}\label{c:hermitian} Suppose that $M\in\mathbb{R}^{4n\times 4n}$ is a $JRS$-symmetric matrix.
\begin{itemize}
\item[$(1)$] If $M$ is also symmetric,
 there exists an orthogonally $JRS$-symplectic matrix $W\in\mathbb{R}^{4n\times 4n}$ such that
\begin{equation}\label{e:H}
WMW^T=H_0\oplus H_0\oplus H_0\oplus H_0,
\end{equation}
where $H_0\in\mathbb{R}^{n\times n}$ is a symmetric tridiagonal matrix \cite{jwl13}.
\item[$(2)$] If $M$ is also skew-symmetric, there exists an orthogonally $JRS$-symplectic matrix $W\in\mathbb{R}^{4n\times 4n}$ such that $WMW^T=H$ has the form \eqref{e:hess}  with
 $H_0=-H_0^T\in\mathbb{R}^{n\times n}$ tridiagonal and $H_1, \ H_2,\ H_3\in\mathbb{R}^{n\times n}$ diagonal.
 \end{itemize}
\end{corollary}

\subsection{The  $JRS$-QR decomposition}
In analogous processing,  we  define and calculate the $JRS$-QR decomposition of $JRS$-symmetric matrices.
\begin{definition}\label{d:uppertriangular} A $JRS$-symmetric matrix  $R\in\mathbb{R}^{4m\times 4n}$ is called an upper $JRS$-triangular matrix if 
\begin{equation}\label{e:qr}
R=\left[
                  \begin{array}{rrrr}
                 R_0 &R_2& R_1 & R_3 \\
                -R_2 & R_0 & R_3 & -R_1 \\
                -R_1 & -R_3 & R_0 & R_2 \\
                -R_3 & R_1 & -R_2 & R_0 \\
                  \end{array}
                \right],
\end{equation}
where $R_0\in\mathbb{R}^{m\times n}$ is upper triangular, $R_1, \ R_2$,  and $R_3\in\mathbb{R}^{m\times n}$ are  strictly upper triangular.
 Moreover,  if  $R_0$ is also strictly upper triangular then $R$ is called a strictly  upper $JRS$-triangular matrix. 
 \end{definition}

\begin{theorem}\label{t:qr} Suppose that $M\in\mathbb{R}^{4m\times 4n}$ is a $JRS$-symmetric matrix.
Then there exists an orthogonally $JRS$-symplectic matrix $W\in\mathbb{R}^{4m\times 4m}$ such that
$W^T\Upsilon_Q=R\in\mathbb{R}^{4m\times 4n}$ is an upper $JRS$-triangular form.
\end{theorem}
\begin{proof}
The theorem can be proved in a similar way with  Theorem \ref{t:hess}.
\end{proof}

Notice that if $n=1$, $W$ acts like a Householder transformation to simultaneously delete nonzero elements of $M$ besides $(1,1),(m+1,2),(2m+1,3),(3m+1,4)$ positions; in this case we denote
\begin{equation}\label{e:house}
W=\texttt{house}(M).
\end{equation}
This notation will be used in the outlines of our algorithms.

\subsection{The real $JRS$-Schur decomposition}
The real $JRS$-Schur form can be introduced for $JRS$-symmetric matrices.
\begin{definition}\label{d:schur} A $JRS$-symmetric matrix  $T\in\mathbb{R}^{4n\times 4n}$ is called the  real $JRS$-Schur form if 
\begin{equation}\label{e:schur}
T=\left[
                  \begin{array}{rrrr}
                 T_0 &T_2& T_1 & T_3 \\
                -T_2 & T_0 & T_3 & -T_1 \\
                -T_1 & -T_3 & T_0 & T_2 \\
                -T_3 & T_1 & -T_2 & T_0 \\
                  \end{array}
                \right],
\end{equation}
where $T_0\in\mathbb{R}^{n\times n}$ is a real Schur form, $T_1, \ T_2$,  and $T_3\in\mathbb{R}^{n\times n}$  are   upper triangular.
 \end{definition}

\begin{theorem}\label{t:schur}
Suppose that $M\in\mathbb{R}^{4n\times 4n}$ is a $JRS$-symmetric matrix.
Then there exists an orthogonally $JRS$-symplectic matrix $W\in\mathbb{R}^{4n\times 4n}$ such that
$W^TMW=T\in\mathbb{R}^{4n\times 4n}$ is a real $JRS$-Schur form.
\end{theorem}
\begin{proof}
The theorem can be proved in a similar way with  Theorem \ref{t:hess}.
\end{proof}
\subsection{The structure-preserving $JRS$-Hessenberg QR iteration}\label{s:JRSHessQR}
Based on the previous structure-preserving decompositions,  we turn to designing a real structure-preserving  algorithm of computing the  real $JRS$-Schur decomposition. Let  $M\in\mathbb{H}^{4n\times 4n}$ be $JRS$-symmetric, then a practical $JRS$-QR algorithm can be written as 
\begin{flushleft}
\begin{verse}
$H=VMV^T$\\
{\bf for} $s=1,2,\ldots$\\
\ \ \ \  $H=WR$\ \ ($JRS$-QR decomposition)\\
\ \ \ \ $H=RW$\\
{\bf end}
\end{verse}
\end{flushleft}
where  each  $V,W\in\mathbb{R}^{4n\times 4n}$ is orthogonally $JRS$-symplectic and  $R\in\mathbb{R}^{4n\times 4n}$ is upper $JRS$-triangular.  
When  $M$  has complex eigenvalue  this real iteration is associated with a difficulty  that $H$ can never converge to $JRS$-triangular form. 
The expectations must be lowered and  we must  be content with the calculation of  an alternative decomposition--the real $JRS$-Schur decomposition. 
If $V$ is chosen so that $H$ is upper $JRS$-Hessenberg, then the amount of work per iteration is reduced from $O(n^3)$ to $O(n^2)$.

The  traditional QR algorithm can be adapted to compute a real $JRS$-Schur form of $M$ in
real arithmetic.
\begin{flushleft}
\begin{verse}
$H=VMV^T$  ($JRS$-Hessenberg reduction)\\
{\bf for} $s=1,2,\ldots$\\
\ \ \ \  Determine a scalar $\kappa$.\\
\ \ \ \  $H-\kappa I=WR$,\ \ ($JRS$-QR decomposition)\\
\ \ \ \ $H=RW+\kappa I$.\\
{\bf end}
\end{verse}
\end{flushleft}
The reduction of $M$ to  $JRS$-Hessenberg form is done in real arithmetic. If the Wilkinson shift $\kappa$ is real, the $JRS$-QR step  results  in a real matrix $H$. If  $\kappa$ is complex, we simultaneously  apply two 
  $JRS$-QR steps, one with shift $\kappa$ and the other with shift $\overline{\kappa}$ to yield a matrix $\widehat{H}$. If
\begin{equation*}
  \breve{W}\breve{R}=(H-\kappa I)(H-\overline{\kappa}I)
\end{equation*}
is the $JRS$-QR decomposition of
 $ (H-\kappa I)(H-\overline{\kappa}I)$, 
then 
\begin{equation*} 
 \widehat{H}=\breve{W}^T H\breve{W}.
\end{equation*}
Since
\begin{equation*}
  (H-\kappa I)(H-\overline{\kappa}I)=H^{2}-2Re(\kappa)H+|\kappa|^{2}I
\end{equation*}
is real,  so are $\breve{W}$ and $\widehat{H}$. The strategy of working with complex conjugate Wilkinson
shifts is so called the \emph{Francis double shift} strategy.
The complex arithmetic can be avoided 
 by forming the matrix $H^{2}-2Re(\kappa)H+|\kappa|^{2}I$, computing its Q-factor $\breve{W}$, and then computing
$\widehat{H} = \breve{W}^{T}H\breve{W}$.
 Unfortunately, 
  the formation of
$H^{2}-2Re(\kappa)H+|\kappa|^{2}I$ requires $O(n^{3})$ operations.
So we have to use a remarkable property
of  $JRS$-Hessenberg matrices to sidestep the formation of $H^2-2Re(\kappa)H+|\kappa|^{2}I$. 
Before turning to this property,  we first consider the uniqueness of the upper $JRS$-Hessenberg reduction.

\subsubsection{The uniqueness of the upper $JRS$-Hessenberg reduction}
Let $M$ be a $JRS$-symmetric matrix of order $4n$ and let $H=WMW^T$ be a orthogonally $JRS$-symplectic reduction of $M$ to upper $JRS$-Hessenberg form. When reducing $M$ to upper $JRS$-Hessenberg form $H$ by a unitary similarity, we must introduce $4(2n-1)(n-1)$ zeros but only $(n-1)(n-2)/2$ free zeros into $M$.   
Notice that an orthogonally $JRS$-symmetric matrix has $n(n-1)/2$ degrees of freedom. Since we must use
$(n-1)(n-2)/2$ of the degrees of freedom to introduce zeros in $M$, we have $n-1$ degrees
of freedom left over in $W$,  just enough to specify the first column of $W$.

\begin{theorem}[\bf Implicit Q Theorem for $JRS$-Hessenberg Form]\label{t:implicitQTheorem}
Suppose that $M$ is a $4n$-by-$4n$ $JRS$-symmetric matrix, 
and 
$U:=[u_1,\ldots,u_{4n}]$ and  $V:=[v_1,\ldots,v_{4n}]$ are orthogonally  $JRS$-symplectic matrices such that 
 $U^TMU=H$ 
and 
$V^TMV=\widehat{H}$ 
 are upper $JRS$-Hessenberg forms defined by \eqref{e:hess}.
Let $r$ denote the smallest positive integer for which $H_0(r,r-1)=0$, with the convention that $r=n$ if $H$ is unreduced. If $[u_1,u_{n+1},u_{2n+1},u_{3n+1}]=[v_1,v_{n+1},v_{2n+1},v_{3n+1}]$, then  $[u_s,u_{n+s},u_{2n+s},u_{3n+s}]$ $=$ $\pm[v_s,v_{n+s},v_{2n+s},v_{3n+s}]$ for $s=2:r$. Moreover, if $r<n$, then  $\widehat{H}_0(r+1,r)=0$.
\end{theorem}
\begin{proof}
Define $W=V^TU$ and  two kinds of  partitioning
$$W:=\left[
                  \begin{matrix}%{rrr}
                  w_1&\cdots&w_{4n}
                   \end{matrix}
                \right]
:=\left[
                  \begin{matrix}%{rrrr}
                 W_0 &W_2& W_1 & W_3 \\
                -W_2 & W_0 & W_3 & -W_1 \\
                -W_1 & -W_3 & W_0 & W_2 \\
                -W_3 & W_1 & -W_2 & W_0 \\
                  \end{matrix}
                \right].$$
Then $W$ is orthogonally  $JRS$-symplectic and  $$[w_1,w_{n+1},w_{2n+1},w_{3n+1}]=[e_1,e_{n+1},e_{2n+1},e_{3n+1}].$$
Denote that $H:=[h_1,\ldots,h_{4n}]$.
The equation $\widehat{H}W=WH$ implies  that 
$$\widehat{H}[w_s,w_{n+s},w_{2n+s},w_{3n+s}]=W[h_s,h_{n+s},h_{2n+s},h_{3n+s}], s=2,\cdots,n.$$
So that 
\begin{eqnarray}
\nonumber [w_s,w_{n+s},w_{2n+s},w_{3n+s}]H_0(s,s-1)
=
\widehat{H}[w_{s-1},w_{n+s-1},w_{2n+s-1},w_{3n+s-1}]\qquad\qquad\\
-[W_{s,1}, W_{s,2}, W_{s,3}, W_{s,4}]
\label{e:gwd} \left[
\begin{smallmatrix}
                 H_0(1:s-1,s-1) &H_2(1:s-1,s-1) & H_1 (1:s-1,s-1) & H_3(1:s-1,s-1)  \\
                -H_2(1:s-1,s-1)  &H_0(1:s-1,s-1)  &H_3 (1:s-1,s-1) & -H_1(1:s-1,s-1)  \\
                -H_1(1:s-1,s-1)  & -H_3(1:s-1,s-1)  & H_0 (1:s-1,s-1) & H_2(1:s-1,s-1)  \\
                -H_3(1:s-1,s-1)  &H_1(1:s-1,s-1)  & -H_2(1:s-1,s-1)  &H_0(1:s-1,s-1)  \\
\end{smallmatrix}
                \right],
\end{eqnarray}
where each $W_{s,t}:=W(:,(t-1)n+1:(t-1)n+s-1)$, $t=1,\ldots,4$.
Since $\widehat{H}$ is upper $JRS$-Hessenberg matrix, 
$$\widehat{H}[w_{s-1},w_{n+s-1},w_{2n+s-1},w_{3n+s-1}]:=\left[
\begin{smallmatrix}
                 \widehat{W}_0(:,s-1) &\widehat{W}_2(:,s-1) & \widehat{W}_1(:,s-1) & \widehat{W}_3(:,s-1) \\
                -\widehat{W}_2(:,s-1)  & \widehat{W}_0(:,s-1)  & \widehat{W}_3 (:,s-1)& -\widehat{W}_1(:,s-1) \\
                -\widehat{W}_1(:,s-1)  & -\widehat{W}_3(:,s-1)  &\widehat{W}_0(:,s-1)&\widehat{W}_2(:,s-1) \\
                -\widehat{W}_3(:,s-1)  & \widehat{W}_1(:,s-1) & -\widehat{W}_2(:,s-1)&\widehat{W}_0(:,s-1)  \\
\end{smallmatrix}
\right]$$
is $JRS$-symmetric, 
where  $\widehat{W}_0(:,s-1)$ has its last $n-s$ entries being zeros and the $s$-th entry nonzero, $\widehat{W}_1(:,s-1)$,$\widehat{W}_2(:,s-1)$ and $\widehat{W}_3(:,s-1)$ have  their last $n-s+1$ entries being zeros.
By introduction on $n$, we can see that
$W_0(:,1:s)$ is upper triangular with nonzero entries on its  diagonal, $W_1(:,1:s)$, $W_2(:,1:s)$ and $W_3(:,1:s)$ are strictly upper triangular.
Thus for $2\le s\le r$,  
$$[w_s,w_{n+s},w_{2n+s},w_{3n+s}]=\pm [e_s,e_{n+s},e_{2n+s},e_{3n+s}].$$
Since $U=VW$, we obtain
$$[u_s,u_{n+s},u_{2n+s},u_{3n+s}]=\pm[v_s,v_{n+s},v_{2n+s},v_{3n+s}],\ s=2,\cdots,r.$$ 
Multiplying   equation \eqref{e:gwd} by $w_r^T$ from the left side,   
there is 
$H_0(r,r-1)=w_r^T\widehat{H}w_{r-1}$,
and then 
$$|H_0(r,r-1)|=|u_r^TV\widehat{H}V^Tu_{r-1}|=|u_r^TMu_{r-1}|=|v_r^TMv_{r-1}|=|\widehat{H}_0(r,r-1)|.$$
 If $r<n$, the structures of $W$ and $H$ implies 
 $$\widehat{H}_0(r,r-1)=e_{r+1}^T\widehat{H}e_r=\pm e_{r+1}^T\widehat{H}We_r=\pm e_{r+1}^TWDe_r=W(r+1,:)D(:,r)$$
 $$= \left[\begin{smallmatrix}W_0(r+1,:)&W_2(r+1,:)&W_1(r+1,:)&W_3(r+1,:) \end{smallmatrix}\right]
 \left[\begin{smallmatrix}
 H_0(:,r)\\
 - H_2(:,r)\\
 - H_1(:,r)\\
 - H_3(:,r)
 \end{smallmatrix}\right]=0.$$
\end{proof}
An important result following the implicit Q theorem  is that if both $U^TMU=H$ and $V^TMV=\widehat{H}$ are  unreduced upper $JRS$-Hessenberg matrices and $[u_1,u_{n+1},u_{2n+1},u_{3n+1}]$ $=$ $[v_1,v_{n+1},$ $v_{2n+1},$ $v_{3n+1}]$, then $H$ and $\widehat{H}$ are ``essentially equal'' in the sense that 
$\widehat{H}=S^{-1}HS$ with  $S={\rm diag}(\pm 1,\ldots, \pm 1)$.

\subsubsection{The double-implicit-shift strategy}\label{s:doubleim}

We now return to our preliminary algorithm and modify it to avoid
the expensive computation of $H^{2}-2Re(\kappa)H+|\kappa|^{2}I$. 
Let $\kappa$ be a complex Francis shift of $H$. 
If we compute the
Q-factor $\breve{W}$ of the matrix $H^{2}-2Re(\kappa)H+|\kappa|^{2}I$ then $\widehat{H} = \breve{W}^{T}H\breve{W}$ is the result of
	applying two steps of the QR algorithm with shifts $\kappa$ and $\overline{\kappa}$. The work of  simultaneously determining $\breve{W}$ and $\widehat{H}$ can be  resolved into five steps:
\begin{itemize}
  \item[1.]  Compute the $1$, $n+1$, $2n+1$ and $3n+1$ columns of $C=H^{2}-2Re(\kappa)H+|\kappa|^{2}I\in\mathbb{R}^{4n\times 4n}$,  and save them into  $F\in\mathbb{R}^{4n\times 4}$.
  \item[2.]  Determine  a Householder transformation  $W_{F}\in\mathbb{R}^{4n\times 4n}$ such that 
  $$W_F^TF=\sigma[e_1,e_{n+1},e_{2n+1},e_{3n+1}],$$
 where each $e_s$ denotes the $s$-th column of the identity matrix and $\sigma\in\mathbb{R}$ is nonnegative.

  \item[3.]  Set $H_F=W_F^{T}HW_F$.
  \item[4.]  Use Householder transformations to reduce $H_F$ to upper $JRS$-Hessenberg form $\widehat{H}$. Call the accumulated transformations $\widehat{W}$.
  \item[5.]  Set $\breve{W}=W_F\widehat{W}$.
\end{itemize}
The key computations are the computation of the  $1$, $n+1$, $2n+1$ and $3n+1$ columns of $C$ and
the reduction of $H_F$ to upper $JRS$-Hessenberg form. 
Because $H$ is upper $JRS$-Hessenberg one can effect the first calculation in $O(1)$ operations and the second in  $O(n^{2})$ operations. We now turn to the details.  For simplicity,  if there is no confusion then a $JRS$-symmetric matrix is represented by its first  block  row, such as 
\begin{equation}\label{e:4H}H:=[H_0,H_2,H_1,H_3].\end{equation}
\begin{remark}\label{r:4blocks}
A $JRS$-symmetric matrix is uniquely determined by its four submatrices on the first row block, and the converse is also true. The structure-preserving transformation on a  $JRS$-symmetric matrix  is equivalent to corresponding transformations on four submatrices on the first row block.
\end{remark}

{\bf Getting started.} Define $C=H^{2}-2Re(\kappa)H+|\kappa|^{2}I:=[C_0,C_2,C_1,C_3]$. 
The computation of the first column of $C_s(s=0,1,2,3)$ requires that
we first compute the scalars $2Re(\kappa)$ and $|\kappa|^{2}$. To do this we need to compute
$\kappa$ firstly.  Define
  a submatrix of $H$ according to   $m=n-1$ as 
$$
H_{mn}=\left[
\begin{array}{rrrr}
               H_0(m:n,m:n) & H_2(m:n,m:n)& H_1(m:n,m:n)& H_3(m:n,m:n)\\
              -H_2(m:n,m:n) &  H_0(m:n,m:n) & H_3(m:n,m:n)& - H_1(m:n,m:n) \\
             - H_1(m:n,m:n) & -H_3(m:n,m:n) &   H_0(m:n,m:n) & H_2(m:n,m:n) \\
              -H_3(m:n,m:n) &  H_1(m:n,m:n) & -H_2(m:n,m:n) &   H_0(m:n,m:n)\\
\end{array}
\right],$$ 
where each $H_s(m:n,m:n)$ denotes the submatrix on $m$ and $n$ rows and columns of $H_s$.
Compute the smallest magnitude eigenvalues of $H_{mn}$, and choose it as the shift $\kappa$.

Define $H^2:=[\widetilde{H}_0,\widetilde{H}_2,\widetilde{H}_1,\widetilde{H}_3]$.  Because $H$ is upper $JRS$-Hessenberg, only the first three components of the first column of 
 $\widetilde{H}_s$ are nonzero, $s=0,\ldots,3$. 
They are calculated by
\begin{flalign}
\begin{split}
[\widetilde{H}_0(1:3,1),\widetilde{H}_2(1:3,1),\widetilde{H}_1(1:3,1),\widetilde{H}_3(1:3,1)]= \qquad   \qquad \qquad \qquad \qquad\qquad \qquad \quad\\
\quad \left[\small
     \begin{array}{cc|}
       h^{(0)}_{11} & h^{(0)}_{12} \\
      h^{(0)}_{21} & h^{(0)}_{22} \\
       0 & h^{(0)}_{32} \\
     \end{array}
          \begin{array}{cc|}
       h^{(2)}_{11} & h^{(2)}_{12} \\
      h^{(2)}_{21} & h^{(2)}_{22} \\
       0 & h^{(2)}_{32} \\
     \end{array}
          \begin{array}{cc|}
       h^{(1)}_{11} & h^{(1)}_{12} \\
      h^{(1)}_{21} & h^{(1)}_{22} \\
       0 & h^{(1)}_{32} \\
     \end{array}
          \begin{array}{cc}
       h^{(3)}_{11} & h^{(3)}_{12} \\
      h^{(3)}_{21} & h^{(3)}_{22} \\
       0 & h^{(3)}_{32} \\
     \end{array}
   \right]
   \left[\small
            \begin{array}{r|r|r|r}
              h^{(0)}_{11} &h^{(2)}_{11}& h^{(1)}_{11}&  h^{(3)}_{11} \\
               h^{(0)}_{21} & h^{(2)}_{21}&  h^{(1)}_{21}&  h^{(3)}_{21}\\ \hline
              -h^{(2)}_{11} & h^{(0)}_{11}&  h^{(3)}_{11}& - h^{(1)}_{11} \\
               -h^{(2)}_{21} & h^{(0)}_{21}&  h^{(3)}_{21}&  -h^{(1)}_{21}\\ \hline
              - h^{(1)}_{11} & -h^{(3)}_{11}&  h^{(0)}_{11}&  h^{(2)}_{11} \\
              - h^{(1)}_{21} & -h^{(3)}_{21}&  h^{(0)}_{21}& -h^{(2)}_{21}\\ \hline
              - h^{(3)}_{11} & h^{(1)}_{11}&  -h^{(2)}_{11}& h^{(0)}_{11} \\
              - h^{(3)}_{21} & h^{(1)}_{21}&  -h^{(2)}_{21}& h^{(0)}_{21}\\ 
            \end{array}
          \right].
\end{split}&
\end{flalign}
 Then  the first column of $C_s$  is 
 \begin{equation}
 c_s=C_s(:,1)=\left[
             \begin{array}{c}
             \widetilde{H}_s(1:3,1)-2Re(\kappa)H_s(1:3,1)+|\kappa|^2I(1:3,sn+1)\\
             0\\
             \vdots \\
             0 \\
             \end{array}
           \right],\  s=0,\cdots,3.
\end{equation}

Now we apply the substitution of $t$ for $2Re(\kappa)$ and $d$ for $|\kappa|^{2}$ to make sure that 
our algorithm works even if the Francis double shifts are real. Specifically, suppose
that the matrix $H_{mn}$ has two smallest  magnitude eigenvalues $\lambda$ and $\mu$. Then
\begin{equation*}
  C=(H-\lambda I)(H-\mu I)=H^{2}-(\lambda+\mu)H+\lambda\mu I=H^{2}-tH+dI.
\end{equation*}
Then we collect the first columns of $C_0,\ldots, C_3$ in
$$F:=[c_0,c_2,c_1,c_3]=\left[
             \begin{array}{cccc}
            f_0&f_2&f_1&f_3\\
             0&0&0&0\\
             \vdots&\vdots&\vdots&\vdots \\
             0&0&0&0 \\
             \end{array}
           \right]$$
           with 
$$ f_s=\widetilde{H}_s(1:3,1)-tH_s(1:3,1)+dI(1:3,sn+1)\in\mathbb{R}^3, s=0,\cdots,3.$$
Observe  that the Household transformation  $W_F$ such that 
$W_F^TF:=\sigma [e_1, 0, 0, 0]$
  can be determined in $O(1)$ flops.

{\bf Reduction back to $JRS$-Hessenberg form.}
Since a similarity transformation with $W_F$ only changes the first, second and third rows and columns of
$H_s$, so that $H_F=W_F^THW_F$ has the form
\begin{equation}\label{e:HF}
H_F=\left[
                  \begin{array}{rrrr}
                 H_0^{F} &H_2^{F} & H_1^{F}  & H_3^{F} \\
                -H_2^{F}  & H_0^{F}  & H_3^{F} & -H_1^{F}  \\
                -H_1^{F}  & -H_3^{F}  & H_0^{F}  & H_2^{F}  \\
                -H_3^{F}  & H_1^{F}  & -H_2^{F}  & H_0^{F} \\
                  \end{array}
                \right]
 \end{equation}
where
\begin{equation*}
H_0^F=\left[
                                    \begin{array}{cccccc}
                                      \times &\times & \times& \times & \times & \times \\
                                           \times &\times & \times& \times & \times & \times \\
                                           \times &\times & \times& \times & \times & \times \\
                                      \times & \times & \times  & \times  & \times & \times  \\
                                      0 & 0 & 0 & \times  & \times  & \times  \\
                                      0 & 0 & 0 & 0 & \times  & \times  \\
                                    \end{array}
                                  \right],~
H_{1,2,3}^F=\left[
                                    \begin{array}{cccccc}
                                      \times &\times & \times& \times & \times & \times \\
                                           \times &\times & \times& \times & \times & \times \\
                                           \times &\times & \times& \times & \times & \times \\
                                      0 & 0 & 0  & \times  & \times & \times  \\
                                      0 & 0 & 0 & 0  & \times  & \times  \\
                                      0 & 0 & 0 & 0 &0 & \times  \\
                                    \end{array}
                                  \right].
\end{equation*}
This matrix can be restored to upper $JRS$-Hessenberg form by the orthogonally $JRS$-symplectic transformations.  The calculation proceeds are as follows:
\begin{eqnarray*}
[H_0^F,H_2^F,H_1^F,H_3^F]\overset{W_{1}}{\Longrightarrow}\qquad \qquad \qquad \qquad \qquad\qquad\qquad\qquad\qquad\qquad\qquad\qquad\qquad\qquad\qquad\qquad\\
\left[\small
\begin{array}{cccccc|cccccc|cccccc|cccccc}%\setlength{\arraycolsep}{0.01pt} 
\times &\times & \times& \times & \times & \times 
&\times &\times & \times& \times & \times & \times 
 &\times &\times & \times& \times & \times & \times 
 &\times &\times & \times& \times & \times & \times\\
\times &\times & \times& \times & \times & \times 
&0 &\times & \times& \times & \times & \times 
 &0 &\times & \times& \times & \times & \times 
 &0 &\times & \times& \times & \times & \times\\
0 &\times & \times& \times & \times & \times 
&0 &\times & \times& \times & \times & \times 
 &0 &\times & \times& \times & \times & \times 
 &0 &\times & \times& \times & \times & \times\\
0 &\times & \times& \times & \times & \times 
&0&\times & \times& \times & \times & \times 
 &0 &\times & \times& \times & \times & \times 
 &0 &\times & \times& \times & \times & \times\\
0&\times & \times& \times & \times & \times 
&0&0 & 0& 0 & \times & \times 
 &0 &0& 0&0& \times & \times 
 &0 &0 & 0& 0 & \times & \times\\
0 &0 & 0& 0& \times & \times 
&0 &0 & 0& 0 & 0& \times 
 &0 &0& 0& 0 & 0 & \times 
 &0 &0 & 0& 0 &0 & \times\\
                                    \end{array}
                                  \right]
\end{eqnarray*}
\begin{eqnarray*}
\overset{W_{2}}{\Longrightarrow}\qquad \qquad \qquad \qquad \qquad\qquad\qquad\qquad\qquad\qquad\qquad\qquad\qquad\qquad\qquad\qquad\\
\left[\small
\begin{array}{cccccc|cccccc|cccccc|cccccc}
% 1row
\times &\times & \times& \times & \times & \times 
&\times &\times & \times& \times & \times & \times 
 &\times &\times & \times& \times & \times & \times 
 &\times &\times & \times& \times & \times & \times\\
%2 r
\times &\times & \times& \times & \times & \times 
&0 &\times & \times& \times & \times & \times 
 &0 &\times & \times& \times & \times & \times 
 &0 &\times & \times& \times & \times & \times\\
%3r
0 &\times& \times& \times & \times & \times 
&0 &0 & \times& \times & \times & \times 
 &0 &0 & \times& \times & \times & \times 
 &0 &0 & \times& \times & \times & \times\\
%4r
0 &0 & \times& \times & \times & \times 
&0&0 & \times& \times & \times & \times 
 &0 &0 &\times& \times & \times & \times 
 &0 &0 &\times& \times & \times & \times\\
0&0 & \times& \times & \times & \times 
&0&0 & \times& \times & \times & \times 
 &0 &0& \times&\times& \times & \times 
 &0 &0 & \times& \times & \times & \times\\
0 &0 & \times& \times& \times & \times 
&0 &0 & 0& 0 & 0& \times 
 &0 &0& 0& 0 & 0 & \times 
 &0 &0 & 0& 0 &0 & \times\\
                                    \end{array}
                                  \right]
 \end{eqnarray*}
\begin{eqnarray*}
\overset{W_{3}}{\Longrightarrow}\qquad \qquad \qquad \qquad \qquad\qquad\qquad\qquad\qquad\qquad\qquad\qquad\qquad\qquad\qquad\qquad\\
\left[\small
\begin{array}{cccccc|cccccc|cccccc|cccccc}
% 1row
\times &\times & \times& \times & \times & \times 
&\times &\times & \times& \times & \times & \times 
 &\times &\times & \times& \times & \times & \times 
 &\times &\times & \times& \times & \times & \times\\
%2 r
\times &\times & \times& \times & \times & \times 
&0 &\times & \times& \times & \times & \times 
 &0 &\times & \times& \times & \times & \times 
 &0 &\times & \times& \times & \times & \times\\
%3r
0 &\times& \times& \times & \times & \times 
&0 &0 & \times& \times & \times & \times 
 &0 &0 & \times& \times & \times & \times 
 &0 &0 & \times& \times & \times & \times\\
%4r
0 &0 & \times& \times & \times & \times 
&0&0 & 0& \times & \times & \times 
 &0 &0 &0& \times & \times & \times 
 &0 &0 &0& \times & \times & \times\\
%5r
0&0 & 0& \times & \times & \times 
&0&0 &0& \times & \times & \times 
 &0 &0& 0&\times& \times & \times 
 &0 &0 & 0& \times & \times & \times\\
%6r
0 &0 & 0& \times& \times & \times 
&0 &0 & 0& \times & \times& \times 
 &0 &0& 0& \times & \times & \times 
 &0 &0 & 0& \times &\times & \times\\
                                    \end{array}
                                  \right]
\end{eqnarray*}
\begin{eqnarray*}
\overset{W_{4}}{\Longrightarrow}\qquad \qquad \qquad \qquad \qquad\qquad\qquad\qquad\qquad\qquad\qquad\qquad\qquad\qquad\qquad\qquad\\
\left[\small
\begin{array}{cccccc|cccccc|cccccc|cccccc}
% 1row
\times &\times & \times& \times & \times & \times 
&\times &\times & \times& \times & \times & \times 
 &\times &\times & \times& \times & \times & \times 
 &\times &\times & \times& \times & \times & \times\\
%2 r
\times &\times & \times& \times & \times & \times 
&0 &\times & \times& \times & \times & \times 
 &0 &\times & \times& \times & \times & \times 
 &0 &\times & \times& \times & \times & \times\\
%3r
0 &\times& \times& \times & \times & \times 
&0 &0 & \times& \times & \times & \times 
 &0 &0 & \times& \times & \times & \times 
 &0 &0 & \times& \times & \times & \times\\
%4r
0 &0 & \times& \times & \times & \times 
&0&0 & 0& \times & \times & \times 
 &0 &0 &0& \times & \times & \times 
 &0 &0 &0& \times & \times & \times\\
%5r
0&0 & 0& \times & \times & \times 
&0&0 &0& 0 & \times & \times 
 &0 &0& 0&0& \times & \times 
 &0 &0 & 0& 0 & \times & \times\\
%6r
0 &0 & 0& 0& \times & \times 
&0 &0 & 0& 0 & \times& \times 
 &0 &0& 0& 0 & \times & \times 
 &0 &0 & 0& 0 &\times & \times\\
                                    \end{array}
                                  \right]
\end{eqnarray*}
\begin{eqnarray*}
\overset{W_{5}}{\Longrightarrow}\qquad \qquad \qquad \qquad \qquad\qquad\qquad\qquad\qquad\qquad\qquad\qquad\qquad\qquad\qquad\qquad\\
\left[\small
\begin{array}{cccccc|cccccc|cccccc|cccccc}
% 1row
\times &\times & \times& \times & \times & \times 
&\times &\times & \times& \times & \times & \times 
 &\times &\times & \times& \times & \times & \times 
 &\times &\times & \times& \times & \times & \times\\
%2 r
\times &\times & \times& \times & \times & \times 
&0 &\times & \times& \times & \times & \times 
 &0 &\times & \times& \times & \times & \times 
 &0 &\times & \times& \times & \times & \times\\
%3r
0 &\times& \times& \times & \times & \times 
&0 &0 & \times& \times & \times & \times 
 &0 &0 & \times& \times & \times & \times 
 &0 &0 & \times& \times & \times & \times\\
%4r
0 &0 & \times& \times & \times & \times 
&0&0 & 0& \times & \times & \times 
 &0 &0 &0& \times & \times & \times 
 &0 &0 &0& \times & \times & \times\\
%5r
0&0 & 0& \times & \times & \times 
&0&0 &0& 0 & \times & \times 
 &0 &0& 0&0& \times & \times 
 &0 &0 & 0& 0 & \times & \times\\
%6r
0 &0 & 0& 0& \times & \times 
&0 &0 & 0& 0 & 0& \times 
 &0 &0& 0& 0 & 0 & \times 
 &0 &0 & 0& 0 &0 & \times\\
                                    \end{array}
                                  \right].
\end{eqnarray*}
Now we prove that the upper $JRS$-Hessenberg structure is preserved through the shift QR iteration.
\begin{theorem}\label{t:prHess} 
Suppose $H\in\mathbb{R}^{4n\times 4n}$ is unreduced upper $JRS$-Hessenberg, 
and $\kappa\in\mathbb{C}$ does not represent an eigenvalue of $H$. If $WR=C:=H^{2}-2Re(\kappa)H+|\kappa|^{2}I$ is a $JRS$-QR decomposition, then $\widehat{H}=W^THW$ is also upper $JRS$-Hessenberg.
\end{theorem}
\begin{proof}
Since $\kappa$ is not an eigenvalue of $H$, $C$ is nonsingular, and so is $R$. 
The orthogonally $JRS$-symplectic matrix $W=CR^{-1}$, and $W^T=W^{-1}=RC^{-1}$.
Since $CH=HC$, $W^THW=RC^{-1}HCR^{-1}=RHR^{-1}$. 
Note that $R$ and $R^{-1}$ are $JRS$-triangular.  As the product of two $JRS$-triangular matrices with a $JRS$-Hessenberg matrix, $\widehat{H}$  is upper $JRS$-Hessenberg.  
\end{proof}
The implicit determination of $\widehat{H}$ from $H$ outlined above bases on  the {\it Francis QR step}, first described by Francis (1961) and then included in the books \cite{gova13,stewart01}.

\subsubsection{Computing the real $JRS$-Schur form} 
 The standard way to solve the dense nonsymmetric eigenproblem is firstly  reducing a  matrix to  the upper Hessenberg form, and  producing the real Schur form by iteration with  the  Francis QR step. 
In this subsection we indicate  how to reduce a real $JRS$-Hessenberg matrix  $H\in\mathbb{R}^{4n\times 4n}$ to a real $JRS$-Schur form $T=W^THW$ with the  orthogonal $JRS$-symplectic matrix $W$.

Denote that $H:=[H_0,H_2,H_1,H_3]$, $W:=[W_0,W_2,W_1,W_3]$ and $T:=[T_0,T_2,T_1,T_3]$.  
\begin{itemize}
 \item Firstly,  find the largest nonnegative integer $q$ and the smallest nonnegative integer $p$ such
        that 
$$H_0=\begin{array}{cc}
\left[\begin{array}{ccc}
H_{11}&  H_{12}&  H_{13} \\
0&  H_{22}&  H_{23} \\
0&  0&  H_{33} \end{array}\right]
&
\begin{array}{l}
p \\
n-p-q\\
q
\end{array}
     \end{array}$$
     \quad    where $H_{33}$ is upper quasi-triangular and $H_{22}$ is unreduced.
 \item   Secondly,  if $q<n$, perform a Francis  $JRS$-QR step  on the unreduced upper $JRS$-Hessenberg matrix  $H_{22}$:
            $$H_{22}=\breve{W}^TH_{22}\breve{W}.$$ 
\end{itemize}
 Let $\epsilon$ denote the machine  precision. 
The calculated real $JRS$-Schur form $\widehat{T}$ has the structure defined by \eqref{e:schur} 
and is orthogonally similar to a $JRS$-symmetric matrix  near to $H$, i.e., 
$$W^T(H+E)W=\widehat{T},$$
where $W$ is orthogonally $JRS$-symplectic,  $E$ is $JRS$-symmetric with  small $\|E\|_2\approx\epsilon\|H\|_2$.
The calculated $\widehat{W}$ is almost orthogonally $JRS$-symplectic in the sense that $\widehat{W}^T\widehat{W}-I=F$  is $JRS$-symplectic and $\|F\|_2\approx\epsilon$.  %Notice that applying the traditional QR algorithm to 

\bigskip

Recall the observation in Theorem \ref{t:JRSrealq} that the structure-preserving decompositions of  $JRS$-symmetric matrices can lead to  the corresponding decompositions of quaternion matrices.  
For instance, the upper $JRS$-Hessenberg form $H$ defined by \eqref{e:hess} is a real counterpart of quaternion matrix $H_0+H_1i+H_2j+H_3k$, which is a quaternion Hessenberg matrix with {\it real subdiagonal elements};  and the orthogonally $JRS$-symplectic matrix $W$ defined by \eqref{e:w} is  a real counterpart of a unitary quaternion matrix $W_0+W_1i+W_2j+W_3k$. The QR,  block-diagonal Schur and Hessenberg decompositions of quaternion matrices can be easily elicited from those of $JRS$-matrices based on Theorem \ref{t:JRSrealq}. 
One of the most important improvements here
 is that  the subdiagonal (or diagonal) entries of Hessenberg and block-diagonal Schur forms  (or $R$-factor) are real numbers, which will greatly enhance the algorithms based on quaternion matrix decompositions.

\section{A  new implicit double shift quaternion QR algorithm}\label{s:QQRA}
In this section, we present a new fast quaternion QR algorithm with applying the real structure-preserving methods.

A strategy to solve  the eigenproblem of a general quaternion matrix $Q\in\mathbb{H}^{n\times n}$ can be described in two steps:
\begin{itemize} 
\item[$(1)$] Calculate the real $JRS$-Schur form \eqref{e:schur}  of the real counterpart 
 $\Upsilon_Q\in\mathbb{R}^{4n\times 4n}$ of $Q$, and then lead to the quasi upper-triangular Schur matrix
 $$T=T_0+T_1i+T_2j+T_3k\in\mathbb{H}^{n\times n},$$
 where $T_0\in\mathbb{R}^{n\times n}$ is a real Schur form, $T_{1}$, $T_{2}$ and $T_{3} \in\mathbb{R}^{n\times n}$ are   upper triangular.
 \item[$(2)$] Solve the eigenproblem of $T$ and backstep for eigen-information of $Q$ under similarity transformations.
 \end{itemize}
We will concentrate into the first step to develop a new version of the practical quaternion QR algorithm in \cite{bbm89}.  
Without causing any confusion, we use the same notation
$$[Q_0,Q_2,Q_1,Q_3]$$
 to represent the quaternion matrix $Q=Q_0+Q_1i+Q_2j+Q_3k,\ Q_0,\ldots, Q_3\in\mathbb{R}^{n\times n}$,  and its real counterpart $\Upsilon_Q\in\mathbb{R}^{4n\times 4n}$. 
See Remark \ref{r:4blocks} for the explanation.

\subsection{Basic quaternion operations}
At first we introduce several unitary quaternion transformations, including four improved Householder-based transformations and one generalized quaternion Givens transformation.
\subsubsection{Improved Householder-based transformations}
Four Householder-based transformations proposed in \cite{bbm89,sabi06,jwl13,lwzz16}   are recalled with slight improvement.

 Given two different quaternion vectors $x=[x_1,\cdots,x_n]^T,~y=[y_1,\cdots,y_n]^T\in\mathbb{H}^n$ with $\|x\|=\|y\|$ and $y^*x\in\mathbb{R}$, 
there exists a quaternion Householder matrix defined by $\mathscr{H}=I-2uu^*,$  
where $u=\frac{y-x}{\|y-x\|}$, such that $\mathscr{H}y=x$; see  \cite{bbm89} and \cite[Theorem 3.1 and Theorem 3.2]{lwzz16}. 
Applying real structure-preserving methods, we can execute four kinds of improved Householder-based transformations:  for any real vector $v\in\mathbb{R}^n$ with  $\|v\|=1$,
\begin{itemize}
\item when $x=\alpha v$  with $\alpha\in\mathbb{H}$ and $|\alpha|=\|y\|$,    
$\mathscr{H}_1:=I-2uu^*,$  
where $u=\frac{y-x}{\|y-x\|},~(\text{proposed in \cite{bbm89}} )$ 
\item  when $x=\|y\|v$,  
$\mathscr{H}_2:=\frac{1}{\xi}(I-uu^*),$  
where 
 $$u=\frac{y-\xi x}{\sqrt{\|y\|(\|y\|+|y^Tv|)}},~
 \xi=\begin{cases}
1 , \qquad |y^{T}v|=0,\\
-\frac{y^{T}v}{|y^{T}v|},~\text{otherwise},
\end{cases}
$$
  (proposed in \cite{sabi06})
\item when $x=\|y\|v$, 
 $\mathscr{H}_3:=(I-2uu^T)G,$  
 where $u=\frac{Gy-x}{\|Gy-x\|}$, 
$G=\texttt{diag}(g_1,g_2,\dots,g_n),$
$$g_\ell=
\begin{cases}
\frac{\overline{y_\ell}}{|y_\ell|},\quad y_\ell \neq0, \\
1,~\text{otherwise},
\end{cases}
$$
(proposed in \cite{jwl13})
\item  and 
when $x=\|y\|v$, % 
 $\mathscr{H}_4:=G\mathscr{H}_1,$ 
where $G=\texttt{diag}(g_1,g_2,\dots,g_n),$
$$g_\ell=
\begin{cases}
\frac{\overline{z_{\ell}}}{|z_{\ell}|},\quad z_{\ell}\neq 0, \\
1,~\text{otherwise}
\end{cases}
\text{with}~z=\mathscr{H}_1y.
$$
(proposed in \cite{lwzz16})
\end{itemize}
\begin{remark}
If $v$ is one column of the identity matrix,  then $\mathscr{H}_2=\mathscr{H}_4=g\mathscr{H}_1$, where $g$ is a unit quaternion scalar which rotates the nonzero element of $\mathscr{H}_1y$ into a positive number.
\end{remark}
\begin{remark}
As pointed by Li et al. \cite{lwzz16},  $\mathscr{H}_1, \ldots, \mathscr{H}_4$ are unitary quaternion matrices and  only  $\mathscr{H}_1$ is Hermitian and reflective. 
\end{remark}
\begin{remark}  Applying the realization of the quaternion operations  in Section \ref{ss:qJRS}, we can execute the quaternion Householder-based transformations in real arithmetic. The  necessary real flops and assignment numbers are listed in Table \ref{t:householder}.
\end{remark}

\begin{table}
\begin{center}
\mbox{ } \caption{Computation amounts and assignment numbers for $H_\ell$ and $H_\ell x$.
}\label{t:householder}
\bigskip
\begin{tabular}{c|llll }
Methods  & \multicolumn{2}{c}{Generate matrix $\mathscr{H}_\ell$}       & \multicolumn{2}{c}{Transformation $\mathscr{H}_\ell x$}       \\
%\hline
               &  assignment     &     real flops           &      assignment       &  real flops  \\
\hline
$\mathscr{H}_1$     &         $8$         &    $8n+19$              &          $2$         &    $80n-4$         \\
$\mathscr{H}_2$     &       $10$        &      $8n+30$      &                $4$        &     $80n+24$      \\
$\mathscr{H}_3$     &       $n+1$        &     $13n+2$                &     $2n+2$        &     $32n$   \\
$\mathscr{H}_4$     &         $10$        &     $8n+30$        &            $4$        &     $80n+24$     \\
\hline
\end{tabular}
\end{center}
\end{table}

\subsubsection{Generalized quaternion Givens transformations}
Janovsk\'a and Opfer extended the   Givens transformation to quaternion valued matrices  in \cite{joap03}.   
Recall \cite[Theorem 3.4]{joap03} that for given nonzero  vector $x=[x_1,x_2]^T\in\mathbb{H}^2$,  define
$$\mathscr{G}_1=\left[\begin{matrix}\overline{c}&s\\-\overline{s}&c\end{matrix}\right],
~\text{with}~s=-\sigma\frac{\overline{x_2}}{\|x\|},~c=\sigma\frac{\overline{x_1}}{\|x\|},~|\sigma|=1,$$
where $\sigma$ is arbitrary in case $x_1$, $x_2$ are linearly dependent over $\mathbb{R}$ or otherwise $\sigma=\frac{\alpha x_1+\beta x_2}{|\alpha x_1+\beta x_2|}\in\mathbb{H}$ with nonzero vector $[\alpha,\beta]^T\in\mathbb{R}^2$,  then $\mathscr{G}_1$ is a unitary matrix and  $\mathscr{G}_1^*x=\sigma[\|x\|,0]^T$. 
Their extension is based on the traditional form of Givens matrix.
 We will define  a new quaternion Givens transformation in a different view from \cite{joap03,joap05}.
\begin{theorem}
\label{t:g-Givens}
Let $x=[x_1\ x_2]^T\in\mathbb{H}^2$ be given with $x_2\ne 0$. Then there exists a generalized Givens matrix 
$\mathscr{G}_2=\left[\begin{matrix}g_{11}&g_{12}\\g_{21}&g_{22}\end{matrix}\right]$ 
such that $\mathscr{G}_2^*x=[\|x\|_2\ 0]^T$. 
A choice of $\mathscr{G}_2$ is
\begin{flalign}\label{e:g-Givens}
\qquad\qquad\qquad
\begin{split}
g_{11}=\frac{x_1}{\|x\|_2},\  g_{21}=\frac{x_2}{\|x\|_2};\qquad\ \ \\
{\rm if} |x_1|\le |x_2|,\  g_{12}=|g_{21}|,\  g_{22}=-|g_{21}|g_{21}^{-*}g_{11}^*; \\
{\rm if} |x_1|> |x_2|,\  g_{22}=|g_{11}|,\  g_{12}=-|g_{11}|g_{11}^{-*}g_{12}^*.\\
\end{split}&
\end{flalign}
\end{theorem}
\begin{proof}
Because $\mathscr{G}_2$ is required to be unitary, we can define 
$$g_{11}=\frac{x_1}{\|x\|_2}, \ g_{21}=\frac{x_2}{\|x\|_2},$$
and $g_{12}, g_{22}$ should satisfy 
\begin{equation}\label{e:g2122}
g_{11}^*g_{12}+g_{21}^*g_{22}=0,\ g_{12}^*g_{12}+g_{22}^*g_{22}=1.
\end{equation}
In order to ensure stability, the selection problem of $g_{12},\ g_{22}$ will be discussed in the following two cases.
\begin{itemize}
\item[(1)]  $|x_1|\le |x_2|$ if and only if $|g_{11}|\le |g_{21}|$.
From  \eqref{e:g2122}, we get 
$$g_{22}=-g_{21}^{-*}g_{11}^*g_{12},\ 1=|g_{12}|^2+|g_{12}|^2|g_{21}^{-*}g_{11}^*|^2.$$
Then we can choose
$$g_{12}=\frac{1}{\sqrt{1+|g_{21}^{-*}g_{11}^*|^2}}=\frac{1}{\sqrt{1+|g_{21}^{-1}|^2|g_{11}|^2}}=\frac{|g_{21}|}{\sqrt{|g_{21}|^2+|g_{11}|^2}}=|g_{21}|,$$
and then 
$$g_{22}=-|g_{21}|g_{21}^{-*}g_{11}^*.$$
\item[(2)]  $|x_1|>|x_2|$ if and only if $|g_{11}|> |g_{21}|$. 
From  \eqref{e:g2122}, we get 
$$g_{12}=-g_{11}^{-*}g_{21}^*g_{22},\ 1=g_{12}^*g_{12}+g_{22}^*g_{22}=|g_{22}|^2|g_{11}^{-1}|^2|g_{21}|^2+|g_{22}|^2.$$
Therefore we can choose
$$g_{22}=\frac{1}{\sqrt{1+|g_{11}^{-1}|^2|g_{21}|^2}}=\frac{|g_{11}|}{\sqrt{|g_{11}|^2+|g_{21}|^2}}=|g_{11}|,$$
and then 
$$g_{12}=-|g_{11}|g_{11}^{-*}g_{21}^*.$$
\end{itemize}
Obviously, $\mathscr{G}_2$ with such structure is unitary. Finally,
$$\mathscr{G}_2^*x=[\|x\|_2, 0]^T.$$
\end{proof}
\begin{remark}
The quaternion Givens matrix $\mathscr{G}_2$ is the generalization of real Givens matrix, and $|g_{11}|=|g_{22}|$, $|g_{21}|=|g_{12}|$.
\end{remark}
\begin{remark} According to the absolute value of $x_1, x_2$, we take the different $g_{12}, g_{22}$.
When $|x_1|\leq |x_2|$, then $|g_{12}|=|g_{21}|\geq\frac{\sqrt{2}}{2}$. It can ensure  stability in the process of computing $g_{22}$.
When $|x_1|> |x_2|$, then $|g_{22}|=|g_{11}|>\frac{\sqrt{2}}{2}$. It can ensure  stability in the process of computing $g_{21}$.
\end{remark}
\begin{remark}
In Table \ref{t:folps4givens},  we present the comparison on the computation amounts and assigment numbers between the generalized quaternion Givens transformations and the fast quaternion Givens transformations.
\end{remark}

\begin{table}%[]
\begin{center}
\mbox{ } \caption{Computation amounts and assignment numbers for quaternion Givens Transformations.
}\label{t:folps4givens}
\bigskip
\begin{tabular}{l|cccc }
Methods  & \multicolumn{2}{c}{Generate $\mathscr{G}$}       & \multicolumn{2}{c}{Givens Transformation $\mathscr{G}^*x$}     \\
%\hline
         &  assignment     &     real flops           &      assignment       &  real flops  \\
\hline
    Fast Quaternion Givens $\mathscr{G}_1$     &         $15$         &    $120$              &       $2$            & $120$      \\
 Generalized Quaternion Givens    $\mathscr{G}_2$ &         $9$             &     $69$           &         $2$               & $120$    \\
\hline
\end{tabular}
\end{center}
\end{table}

\subsection{The quaternion Hessenberg reduction}
The Hessenberg reduction of quaternion matrices  based on quaternion Householder-based transformations were firstly proposed in \cite{bbm89}  in the range of our knowledge. 

Reducing a  quaternion matrix $Q\in\mathbb{H}^{n\times n}$ to the  Hessenberg form  means to find a unitary quaternion matrix $W=W_0+W_1i+W_2j+W_3k$ such that 
 \begin{equation}\label{e:hessQ_WH}W^*QW=H,\end{equation}
where  $H=H_0+H_1i+H_2j+H_3k$, $H_0,\ldots, H_3\in\mathbb{R}^{n\times n}$ are upper Hessenberg  matrices. 
Since the real  counterpart of  $Q$ is $JRS$-symmetric, we can firstly  calculate the $JRS$-Hessenberg form $H$ of $\Upsilon_Q$ as shown in the proof of Theorem \ref{t:hess}, and then backstep for the Hessenberg form of the quaternion matrix by Theorem \ref{t:JRSrealq}.

Now we present three real structure-preserving  algorithms.
For simplicity, we need to define  two auxiliary functions: 
\begin{equation}\label{e:idin} 
\texttt{id}(p)=[p,n+p,2n+p,3n+p],
~\texttt{in}(p,q)=[p:q,n+p:n+q,2n+p:2n+q,3n+p:3n+q]\end{equation}
 for any positive integers $p$ and $q$.
\begin{algorithm}[\bf Quaternion Hessenberg  Reduction Based on $\mathscr{H}_1$]\label{a:hessreduction1}  
Given a 
quaternion matrix $Q=Q_0+Q_1i+Q_2j+Q_3k\in\mathbb{H}^{n\times n}$, this algorithm overwrites $Q$ with an upper Hessenberg quaternion matrix $H=H_0+H_1i+H_2j+H_3k$ satisfying  $H=W^*QW$, where $W$ is a unitary quaternion matrix.
\begin{enumerate}
\item[$1.$] Form $H=[Q_0;Q_1;Q_2;Q_3]$; 
\item[$2.$] for s=2:n-1   
\item[$3.$] \qquad $[u,~\beta]=\mathscr{H}_1(H(\texttt{in}(s,n),s-1))$;
\item[$4.$] \qquad $Y=H(\texttt{in}(s,n),s-1:n)$;
\item[$5.$]  \qquad $H(\texttt{in}(s,n),s-1:n)=Y-(\beta*u)*(u'*Y)$;
\item[$6.$] \qquad $Y=[H(1:n,s:n),-H(n+1:2n,s:n),-H(2n+1:3n,s:n),$...
\item[] \qquad\qquad $-H(3n+1:4n,s:n)]$;
\item[$7.$] \qquad $Y=Y-(Y*u)*(\beta*u')$;
\item[$8.$]  \qquad $H(:,s:n)=[Y(1:n,1:n+1-s);-Y(1:n,nn+1:2(n+1-s));$...
\item[] \qquad\qquad $-Y(1:n,2(n+1-s)+1:3(n+1-s));-Y(1:n,3(n+1-s)+1:4(n+1-s))]$;
\item[$9.$] end 
\end{enumerate}

\end{algorithm}

\begin{algorithm}[\bf Quaternion Hessenberg  Reduction Based on $\mathscr{H}_2$ or $\mathscr{H}_4$ ]\label{a:hessreduction2}  Given a 
quaternion matrix $Q=Q_0+Q_1i+Q_2j+Q_3k\in\mathbb{H}^{n\times n}$, this algorithm overwrites $Q$ with an upper Hessenberg quaternion matrix $H=H_0+H_1i+H_2j+H_3k$ satisfying  $H=W^*QW$, where $W$ is a unitary quaternion matrix.
\begin{enumerate}
  \item[$1.$] Run Algorithm \ref{a:hessreduction1}; and store the computed upper Hessenberg matrix as  $H:=[H_0,H_2,H_1,H_3]$;
   \item[$2.$] for s=2:n
   \item[$3.$] \qquad\qquad $G=\texttt{JRSGivens}(H(\texttt{id}(s+1),s));$\quad ({\rm see~\cite[Algorithm~3.3]{jwl13}})
\item[$4.$]   \qquad \qquad          $[H_0(t,s:n), H_2(t,s:n), H_1(t,s:n),H_3(t,s:n)]$
\item[       ]\qquad \qquad \qquad \ \      $=G^T[H_0(t,s:n);-H_2(t,s:n);-H_1(t,s:n);-H_3(t,s:n)];$
\item[$5.$]    \qquad \qquad         $[H_0(:,t),H_2(:,t),H_1(:,t),H_3(:,t)]=[H_0(:,t),H_2(:,t),H_1(:,t),H_3(:,t)]G;$
\item[$7.$] end
\end{enumerate}
\end{algorithm}

\begin{algorithm}[\bf Quaternion Hessenberg  Reduction Based on $\mathscr{H}_3$]\label{a:hessreduction} Given a 
quaternion matrix $Q:=[Q_0,Q_2,Q_1,Q_3]$, where $Q_{0,1,2,3}\in\mathbb{R}^{n\times n}$, this algorithm overwrites $Q$ with an upper Hessenberg quaternion matrix $H:=[H_0,H_2,H_1,H_3]$  satisfying $\Upsilon_H=\Upsilon_W^T\Upsilon_Q\Upsilon_W$, where $W:=[W_0,W_2,W_1,W_3]$ is a unitary quaternion matrix.
\begin{itemize}\setlength{\itemsep}{1pt}
\item[$1.$]  for s=1:n-1
\item[$2.$] \qquad for t=s+1:n
\item[$3.$] \qquad \qquad            $G=\texttt{JRSGivens}(Q_0(t,s),Q_1(t,s),Q_2(t,s),Q_3(t,s));\quad ({\rm see~\cite[Algorithm~3.3]{jwl13}})$
\item[$4.$]   \qquad \qquad          $[Q_0(t,s:n), Q_2(t,s:n), Q_1(t,s:n),Q_3(t,s:n)]$
\item[       ]\qquad \qquad \qquad \ \      $=G^T[Q_0(t,s:n);-Q_2(t,s:n);-Q_1(t,s:n);-Q_3(t,s:n)];$
\item[$5.$]    \qquad \qquad         $[Q_0(:,t),Q_2(:,t),Q_1(:,t),Q_3(:,t)]=[Q_0(:,t),Q_2(:,t),Q_1(:,t),Q_3(:,t)]G;$
\item[$6.$]  \qquad       end
\item[$7.$]  \qquad       if $s<n-1$
\item[$8.$] \qquad \qquad       $ [u,\beta]=\texttt{house}(Q_0(s+1:n,s));$
\item[$9.$]  \qquad \qquad        $Q_{0,1,2,3}(s+1:n,s:n)=(I-\beta uu^T)Q_{0,1,2,3}(s+1:n,s:n);$
\item[$10.$]   \qquad \qquad      $Q_{0,1,2,3}(:,s+1:n)=Q_{0,1,2,3}(:,s+1:n)(I-\beta uu^T);$
\item[$11.$]   \qquad      end       
\item[$12.$]    end   
\end{itemize}
\end{algorithm}
In line $3$ of Algorithm \ref{a:hessreduction2} and Algorithm \ref{a:hessreduction}, running the function $\texttt{JRSGivens}$ costs  $11$ flops including in $1$ square root operation. The transformation $G$ acts
as a four-dimensional Givens rotation \cite{fmm01}. We refer to   \cite{mmt03,tiss01}  for a backward stable
implementation of  the generalized symplectic Givens rotation \eqref{e:gl} and more Givens-like actions.
 
\begin{remark}  With the same aim of executing the quaternion Hessenberg reduction  in real arithmetic,  Algorithms \ref{a:hessreduction1}  and  \ref{a:hessreduction} are respectively   based on the Householder-based transformations $\mathscr{H}_1$ and $\mathscr{H}_3$. 
 The marked difference between them   is in the following  two aspects.
\begin{itemize}
\item They utilize different real counter parts of quaternion matrices:    the real counter part used in  Algorithm \ref{a:hessreduction} is defined as in \eqref{e:realq},  while that in Algorithm \ref{a:hessreduction1}  is defined as 
  \begin{equation}\label{e:realq1}
   \widehat{\Upsilon}_Q\equiv\left[
              \begin{array}{rrrr}
                Q_0 & -Q_1&- Q_2 &- Q_3 \\
                Q_1 & Q_0 & -Q_3 & Q_2 \\
                Q_2 & Q_3 & Q_0 & -Q_1 \\
                Q_3 & -Q_2 & Q_1 & Q_0 \\
              \end{array}
            \right].
\end{equation}
These two real counter parts  are similar to each other and have the same  functionality. 
\item They adopt different styles of data motion:
the loads and stores of data are transported by  four $n$-by-$n$ matrices  in  Algorithm \ref{a:hessreduction},
while in Algorithm \ref{a:hessreduction1} by one  $4n$-by-$n$ matrices. 
\end{itemize}
\end{remark}

\begin{remark} 
 Algorithms \ref{a:hessreduction1}-\ref{a:hessreduction} are real structure-preserving methods with calculating the quaternion Hessenberg matrix defined in \cite{bbm89}. 
The calculated quaternion Hessenberg matrix by Algorithm \ref{a:hessreduction1} as well as that in \cite{bbm89} has quaternion elements on the subdiagonal; meanwhile,  the calculated quaternion Hessenberg matrices by Algorithms \ref{a:hessreduction2}  and \ref{a:hessreduction} have positive real numbers on the subdiagonals. 
Algorithm \ref{a:hessreduction2}  is the same as Algorithm \ref{a:hessreduction1}  but with an additional step of rotating the quaternion elements on the subdiagonal  to positive real numbers. 
Computation amounts numbers for the Hessenberg reduction of dense quaternion matrices are listed in the first two columns of Table \ref{t:hess_rehess}.

\begin{table}%[]
\begin{center}
\mbox{ } \caption{Computation amounts and assignment numbers for Hessenberg reduction 
}\label{t:hess_rehess}
\bigskip
\begin{tabular}{l|llll }
Householder  & \multicolumn{2}{c}{Dense Matrix $Q$}       & \multicolumn{2}{c}{Broken Hessenberg matrix $H_F$ }     \\
%\hline
               &  assignment     &     real flops           &      assignment       &  real flops  \\
\hline
$\mathscr{H}_1$  &   $9n-12$      &${128}n^3/{3}$&  8n-9    &$188n^2$\\
$\mathscr{H}_2$ or $\mathscr{H}_4$&     $14n-17$      &  $ {184}n^3/{3}$&    $13n-14$     &$272n^2$\\
$\mathscr{H}_3$  &      $8n^2+5n-26$   & ${80}n^3/{3}$&  $64n-121$    &$128n^2$\\
\hline
\end{tabular}
\end{center}
\end{table}

\end{remark}

\begin{remark}  In Algorithm \ref{a:hessreduction1} (lines 4-7),  we have improved the line 3 of Algorithm 4.5 in \cite{lwzz16} for multiplication  by  Householder matrices by reducing data motion. 
Remind that data motion is an important factor
when reasoning about performance.
\end{remark}

\subsection{Quaternion Hessenberg  QR}
According to the conventional QR iteration method, the practical QR algorithm of quaternion matrices can be presented as 
\begin{algorithm}[Practical Quaternion QR Algorithm] \label{a:practicalQQR}
 Input 
quaternion matrix $Q=Q_0+Q_1i+Q_2j+Q_3k\in\mathbb{H}^{n\times n}$.
\begin{itemize}
\item[$1.$]  Preliminarily reduce $Q$ to the  Hessenberg form $H$ (e.g., by Algorithm \ref{a:hessreduction}).

\item[$2.$] Until convergence, run
\item[] \qquad \qquad {\rm Factor}\  $H=WR$;  
\item[]\qquad \qquad {\rm Set}\ $H=RW$.
\end{itemize}
\end{algorithm}
\noindent
In general case, the subdiagonal entries of $H$ tends to zero when  proceeding the iteration.
The main work is the QR factorization of the upper Hessenberg matrix $H$.

Now we reduce a quaternion Hessenberg  matrix  into a triangular quaternion matrix by  unitary transformations based on the generalized quaternion Givens  matrices. 
\begin{algorithm}[\bf Quaternion Hessenberg  QR]\label{a:HessQRggivens}
 Given an upper Hessenberg quaternion matrix $H:=[H_0,H_2,H_1,H_3]$, where $H_{0,1,2,3}\in\mathbb{R}^{n\times n}$, the following algorithm overwrites $H$ with an upper triangular quaternion matrix $R:=[R_0,R_2,R_1,R_3]$ which satisfies  $\Upsilon_R=\Upsilon_W^T\Upsilon_H$, where $W:=[W_0,W_2,W_1,W_3]$ is a unitary quaternion matrix.
 \begin{itemize}\setlength{\itemsep}{1pt}
\item[1.]for s=1:n-1
\item[2.] \qquad $x:=H([s,s+1],[s,2*n+s,n+s,3*n+s]$;                              
\item[3.] \qquad     calculate the generalized quaternion  Givens matrix $\mathscr{G}_2$ as in Theorem \ref{t:g-Givens};
\item[4.] \qquad $H([s,s+1],[s,2*n+s,n+s,3*n+s]=\mathscr{G}_2^**H([s,s+1],[s,2*n+s,n+s,3*n+s]$;
\item[5.] end
\end{itemize}
\end{algorithm}
In Algorithm \ref{a:HessQRggivens}, $n-1$ generalized quaternion Givens matrices are calculated.  It needs $69$ real flops and $3$ square root operations to generate each $\mathscr{G}_2$ by equation \eqref{e:g-Givens}  if $x_1$ and $x_2$ are quaternion numbers.
Notice that if  $x_2$ is real,  
at most $48$ flops (at least $33$ flops) can be saved. 
This means  if the inputting quaternion Hessenberg matrix has real subdiagonal entries (i.e., $H_0$ is of upper Hessenberg form and $H_{1,2,3}$
 are upper triangular), then the amount of calculation  can be saved.  
So the cost of Algorithm \ref{a:HessQRggivens} is about $120n^2$ for a quaternion Hessenberg matrix of order $n$. 
If we  use  fast quaternion Givens transformations  instead of the generalized quaternion Givens transformations 
in line 3 of Algorithm \ref{a:HessQRggivens},  the cost of per iteration will rise to  about $148n^2$ for a quaternion Hessenberg matrix of order $n$.

\subsection{The implicit double shift quaternion QR algorithm}\label{ss:qralg}
To ensure rapid convergence of quaternion QR algorithm, we need to shift the eigenvalue. 
Bunse-Gerstner, Byers and Mehrmann \cite{bbm89} pointed that  the single-shift technique cannot choose any nonreal quaternion  as the shift  because of noncommunity of quaternions and  directly  proposed the implicitly double shift QR algorithm. 
They proposed the implicitly double shift QR algorithm directly. 
 \begin{algorithm}[ Implicitly Double Shift Quaternion QR Algorithm \cite{bbm89}]\label{a:QFQR} 
Given  a quaternion matrix $A\in\mathbb{H}^{n\times n}$,
 set $A_0:=U_0^*AU_0$ where $U_0$ is unitary chosen so that $A_0$ is Hessenberg. \\
For $s=0,1,2,\ldots$
 \begin{itemize}\setlength{\itemsep}{1pt}
 \item[1.] Select an approximate eigenvalue $\mu\in\mathbb{H}$.
 \item[2.] Set $A_{k+1}:=Q_k^*A_kQ_k$ where $Q_k$ is unitary chosen so that $Q_k^*(A_k^2-(\mu+\bar{\mu})A_k+\mu\bar{\mu})$ is triangular.
 \end{itemize}
 \end{algorithm}
\noindent
Generally, the $A_k^2-(\mu+\bar{\mu})A_k+\mu\bar{\mu}$ can not be explained as $(A_k-\mu I)(A_k-\bar{\mu} I)$ when the shift $\mu$ is a nonreal quaternion number.

In this section,  
we firstly introduce the implicitly double shift $JRS$-QR algorithm for calculating  real $JRS$-Schur forms of real counterparts of quaternion matrices, and then propose a new and fast implicit double shift quaternion QR algorithm.  
Based on the real structure-preserving methods,
the double shift technique is applied to the real counterpart instead of quaternion matrix itself and the dimension is not expanded. 

\subsubsection{ The implicitly double shift $JRS$-QR algorithm}
Once the upper Hessenberg reduction is completed,  the calculation of 
the real $JRS$-Schur form by the Francis QR step  becomes the main step of solving  the dense unsymmetric eigenproblem.

Firstly, we present the Francis $JRS$-QR step on the unreduced upper $JRS$-Hessenberg matrix $H$. 
\begin{algorithm}[\bf Francis $JRS$-QR step]\label{a:Francisqrstep}
Given the unreduced upper $JRS$-Hessenberg matrix $H\in\mathbb{R}^{4n\times 4n}$ and $s,t\in\mathbb{R}$, 
 this algorithm overwrite $H$ with $W_F^THW_F$, where $W_F$  is a orthogonal $JRS$-symplectic matrix. 
\begin{itemize}\setlength{\itemsep}{1pt}
\item[1.] m=n-1;
\item[2.] F=H(\texttt{in}(1,3),:)*H(:,\texttt{id}(1))-s*H(\texttt{in}(1,3),\texttt{id}(1))+t*[[1;0;0],0,0,0]; (see definitions in \eqref{e:idin})
\item[3.] for k=1:n-2  
\item[4.]\qquad       $W_F$= \texttt{house}(F);    ( the  function  \texttt{house} is defined by \eqref{e:house} )
\item[5.]\qquad        q=\texttt{max}(1,k-1);
\item[6.]\qquad       H(\texttt{in}(k,k+2),\texttt{in}(q,n))=$W_F^T$ *H(\texttt{in}(k,k+2),\texttt{in}(q,n));
\item[7.]\qquad        r=\texttt{min}(k+3,n);
\item[8.]\qquad       H(\texttt{in}(1,r),\texttt{in}(k,k+2))= H(\texttt{in}(1,r),\texttt{in}(k,k+2))*$W_F$;
\item[9.]\qquad        if $k<n-2$ 
\item[10.]\qquad\qquad            F=H(\texttt{in}(k+1,k+3),\texttt{id}(k));
\item[11.]  \qquad      end
\item[12.]   end
\item[13.]  $W_F$= \texttt{house}(H(\texttt{in}(n-1,n),\texttt{id}(n-2)));  
\item[14.]  H(\texttt{in}(n-1,n),\texttt{in}(n-2,n))=$W_F^T$ *H(\texttt{in}(n-1,n),\texttt{in}(n-2,n));
\item[15.]  H(\texttt{in}(n-2,n),\texttt{in}(n-1,n))= H(\texttt{in}(n-2,n),\texttt{in}(n-1,n))*$W_F$;
\item[16.]  $W_F$= \texttt{house}(H(\texttt{id}(n),\texttt{id}(n-1))); 
\item[17.]   H(\texttt{id}(n),\texttt{in}(n-1,n))=$W_F^T$ *H(\texttt{id}(n),\texttt{in}(n-1,n));
\item[18.]    H(\texttt{in}(n-1,n),\texttt{id}(n))=H(\texttt{in}(n-1,n),\texttt{id}(n))*$W_F$;
\end{itemize}
\end{algorithm}    
This algorithm requires $138n^2$ flops. If  $W_F$ is accumulated  into a given orthogonal matrix,  additional $138n^2$ flops are necessary. 
Steps 16-18 are  to delete the nonzero $(n,n-1)$-element of $H_1$, $H_2$ and $H_3$. 
Algorithm \ref{a:Francisqrstep} can preserve the upper  $JRS$-Hessenberg form defined by \eqref{d:hess}.
Notice that if we use  the MATLAB order \texttt{hess} on $M$, the resulted Hessenberg form is not $JRS$-symmetric.

\begin{remark}
In Algorithm \ref{a:Francisqrstep}, we are in essence processing the Hessenberg reduction of the broken quaternion Hessenberg matrix, of which the submatrix of first four rows and three columns no longer has upper Hessenberg form. 
Since only two elements are need to be cancelled, the Householder matrix is $3$-by-$3$, and so the  processing totally  needs  $O(n^2)$ flops.
  The computational counts are listed in the last two columns of Table \ref{t:hess_rehess}.
\end{remark}

During the iteration in  Francis $JRS$-QR step,  it is necessary to monitor the subdiagonal elements in $H_0$ in order to spot any possible decoupling.
We illustrate how to do this  in the following algorithm.
\begin{algorithm}[\bf Real $JRS$-Schur form of a real upper $JRS$-Hessenberg matrix]\label{a:realschur}
Given a real upper $JRS$-Hessenberg matrix $H\in\mathbb{R}^{4n\times 4n}$ and a tolerance \texttt{tol} greater than the unit roundoff,  
this algorithm computes the real $JRS$-Schur canonical form  $W^THW=T$, where $W$ is orthogonally $JRS$-symplectic. 
\begin{itemize}
\item[1.]  {\bf while} $q<n$
\item[2.] \quad Set to zero all subdiagonal elements of $H_0=H(1:n,1:n)$ that satisfy: 
           $$|H_0(i,i-1)|<\texttt{tol}(\|H(i,\texttt{id}(i-1))\|_2+\|H(i-1,\texttt{id}(i))\|_2);$$
 \item[3.] \quad Find the largest nonnegative integer $q$ and the smallest non-negative integer $p$ such
        that 
$$H_0=\begin{array}{cc}
\left[\begin{array}{ccc}
H_{11}&  H_{12}&  H_{13} \\
0&  H_{22}&  H_{23} \\
0&  0&  H_{33} \end{array}\right]
&
\begin{array}{l}
p \\
n-p-q\\
q
\end{array}
     \end{array}$$
     \quad    where $H_{33}$ is upper quasi-triangular and $H_{22}$ is unreduced.
 \item[4.] \quad  If $q<n$, perform a Francis  $JRS$-QR step (Algorithm \ref{a:Francisqrstep}) on the unreduced upper $JRS$-Hessenberg matrix  $H(\texttt{in}(p+1,n-q),\texttt{in}(p+1,n-q))$:
     \begin{eqnarray*}   
                H(\texttt{in}(p+1,n-q),\texttt{in}(p+1,n-q))&=&W_F^TH(\texttt{in}(p+1,n-q),\texttt{in}(p+1,n-q))W_F,\\
           H(1:p,\texttt{in}(p+1:n-q))&=&H(1:p,\texttt{in}(p+1:n-q))W_F,\\
           H(p+1:n-q,\texttt{in}(n-q+1,n))&=&W_F^TH(p+1:n-q,\texttt{in}(n-q+1,n)).
\end{eqnarray*}
\item[5.]  {\bf end}
\end{itemize}
\end{algorithm}
Based on the empirical observation that average only two Francis iterations are required before the lower $1$-by-$1$ or $2$-by-$2$ decouples, 
this algorithm approximately requires  $106\frac{2}{3}n^3$ flops if only the eigenvalues are desired.
 If $W$ and $T$ are computed,  then $325\frac{1}{3}n^3$ flops are necessary. 
 \begin{remark}
 If we use the traditional Francis QR step instead of the Francis  $JRS$-QR step in line 4, then the flops count for computing $T$ and $W$ will rise to  $1600n^3$.  It is worse that $W$ and $T$ will  no longer be $JRS$-symmetric and the storage space will be multiplied four times.
 \end{remark}

\subsubsection{Implicitly Double Shift Quaternion QR Algorithm}
Based on Theorem \ref{t:JRSrealq}, we can develop an implicit double shift quaternion QR algorithm with the help of the $JRS$-symmetric theory and algorithms.  
\begin{algorithm}[\bf Implicitly Double Shift Quaternion QR Algorithm]\label{a:schurform}
Given a quaternion matrix $Q:=[Q_0,Q_2,Q_1,Q_3]$, where $Q_{0,1,2,3}\in\mathbb{R}^{n\times n}$, the following algorithm overwrites $Q$ with the  quasi upper-triangular  Schur matrix  $T:=[T_0,T_2,T_1,T_3]$ which satisfies  $T=W^*QW$, where $W:=[W_0,W_2,W_1,W_3]$ is a unitary quaternion matrix.
\begin{itemize}
\item[1.] Apply Algorithm \ref{a:hessreduction}  to calculate the Hessenberg  form  $\widehat{W}^*Q\widehat{W}=H:=[H_0,H_2,H_1,H_3]$ of the quaternion matrix $Q$, where $\widehat{W}:=[\widehat{W}_0,\widehat{W}_2,\widehat{W}_1,\widehat{W}_3]$ is a unitary quaternion matrix.
\item[2.] Utilize Algorithm \ref{a:realschur} to calculate the  quasi upper-triangular Schur canonical form  $\widetilde{W}^*H\widetilde{W}=T:=[T_0,T_2,T_1,T_3]$ of the quaternion Hessenberg  matrix $H$, where $\widetilde{W}:=[\widetilde{W}_0,\widetilde{W}_2,\widetilde{W}_1,\widetilde{W}_3]$ is a unitary quaternion matrix.
\item[3.] Calculate $W=\widehat{W}*\widetilde{W}$.
\end{itemize}
\end{algorithm}

\begin{remark}Bunse-Gerstner, Byers and Mehrmann \cite{bbm89} straightly suggested to  replace  $H$ by $M=H^2-(\kappa +\overline{\kappa})H+\kappa \overline{\kappa}I$ in the quaternion QR step.  The supporting  theory is applying two steps of shifted QR iteration applied to the real counterpart  $\Upsilon_H$, which is $JRS$-symmetric; see Section \ref{s:JRSHessQR}.  
Since   $\kappa +\overline{\kappa}$ and $\kappa \overline{\kappa}$ are real, if $(\lambda, x)$ is an eigenpair of $H$ then $(\lambda^2-(\kappa +\overline{\kappa})\lambda+\kappa \overline{\kappa}, x)$  is an eigenpair of $M$. 
\end{remark}
\begin{remark}
The eigenvectors of the original quaternion matrix $Q$ can be found by computing the eigenvectors of the  quasi upper-triangular Schur  matrix $T$ produced by Algorithm \ref{a:schurform},  and  transforming them back under
the unitary quaternion transformation $W$. 
Thus the problem of finding the eigenvectors of the
original quaternion matrix $Q$ is reduced to computing the eigenvectors of a quasi-triangular quaternion matrix
$T$.  We will study this project in further.
\end{remark}

The main differences between  Algorithm  \ref{a:schurform}  and Algorithm A5 in \cite{bbm89}  are as follows.
\begin{itemize}
\item[(1)]By Algorithm \ref{a:schurform}, the calculated Hessenberg matrix $H$ in step 1 has real subdiagonal entries, and this structure is preserved  in step 2 (see steps 4-5 in Algorithm \ref{a:realschur}); and hence, the subdiagonal entries of the resulted quasi upper-triangular Schur form are real.  
 The subdiagonal entries of the calculated Hessenberg  form by  Algorithm A5 in \cite{bbm89}  are not necessary to be real.
\item[(2)]In  Algorithm \ref{a:schurform}, the smallest magnitude eigenvalues of the $2$-by-$2$ right-down submatrix  of the unreduced Hessenberg quaternion matrix and its conjugate are chosen  as the double shifts, while the last diagonal element and its conjugate are chosen in \cite[Algorithm A5]{bbm89}.
\item[(3)] The calculation of Algorithm \ref{a:schurform} is only in real arithmetic, while  Algorithm A5 in \cite{bbm89}  runs in quaternion operations.
\end{itemize}

\section{Numerical experiment}\label{s:numex}
In this section we present four numerical examples to compare the efficiency of newly proposed algorithms
 with the state-of-the-art algorithms.
All  numerical experiments are performed on a personal computer with   
 2.4GHz  Intel Core i7 and  8GB 1600 MHz DDR3, 
 and all codes are written in MATLAB using MATLAB version 9.0.0.321247 (2016a).

\begin{example}[\bf Upper Hessenberg  Reduction of Quaternion Matrices]\label{ex:hessQ}
Suppose that 
$$M=M_0+M_1i+M_2j+M_3k:=[M_0,M_2,M_1,M_3]$$
 is a Toeplitz quaternion matrix, where $M_{0,1,2,3}$ are real matrices of order $n$, generated by the Matlab order \texttt{teoplitz} as  
$$M_0=\texttt{teoplitz}(C,R),
M_1=\texttt{teoplitz}(R),
M_2=\texttt{teoplitz}(C),
M_3=\texttt{teoplitz}(R,C),$$ 
with $C=[n, 1:n, n]$ and  $R=C(n:-1:1)$.
For n=100:100:2000, we compare the numerical efficiency of  the following algorithms on Hessenberg reduction:
\begin{itemize}
\item   \texttt{hessq}: Algorithm A3 in \cite{bbm89} based on the quaternion Householder-based transformation(\cite[Algorithms A2]{bbm89}) and using quaternion toolbox \cite{steve-toobox};
\item \texttt{hessQH1}:  based  on the  quaternion Householder-based transformation $\mathscr{H}_1$ in \cite{lwzz16};   
\item \texttt{hessQH2}: based on the quaternion Householder-based transformation $\mathscr{H}_2$ or $\mathscr{H}_4$ in \cite{lwzz16};
\item \texttt{hessQH1im}: Algorithm \ref{a:hessreduction1};  
\item\texttt{hessQH2im}: Algorithm \ref{a:hessreduction2}; 
\item\texttt{hessQH3}: Algorithm \ref{a:hessreduction}. 
\end{itemize}
In the left figure of  Figure \ref{f:hess}, the CPU times costed by six algorithms  are for the calculation of the Hessenberg  form $H:=[H_0,H_2,H_1,H_3]$ 
  and the unitary matrix $W:=[W_0,W_2,W_1,W_3]$. 
In  the right figure of  Figure \ref{f:hess},  the relative error is  defined as
\begin{equation*}\label{e:residual}
Re=\frac{\|\texttt{tril}(H_0,-2)\|_F+\sum_{s=1}^3\|\texttt{tril}(H_s,-1)\|_F}{\|[H_0,H_2,H_1,H_3]\|_F}.
\end{equation*} 
\begin{figure}\label{f:hess}
  \begin{center}
\includegraphics[width=6cm,height=4.5cm]{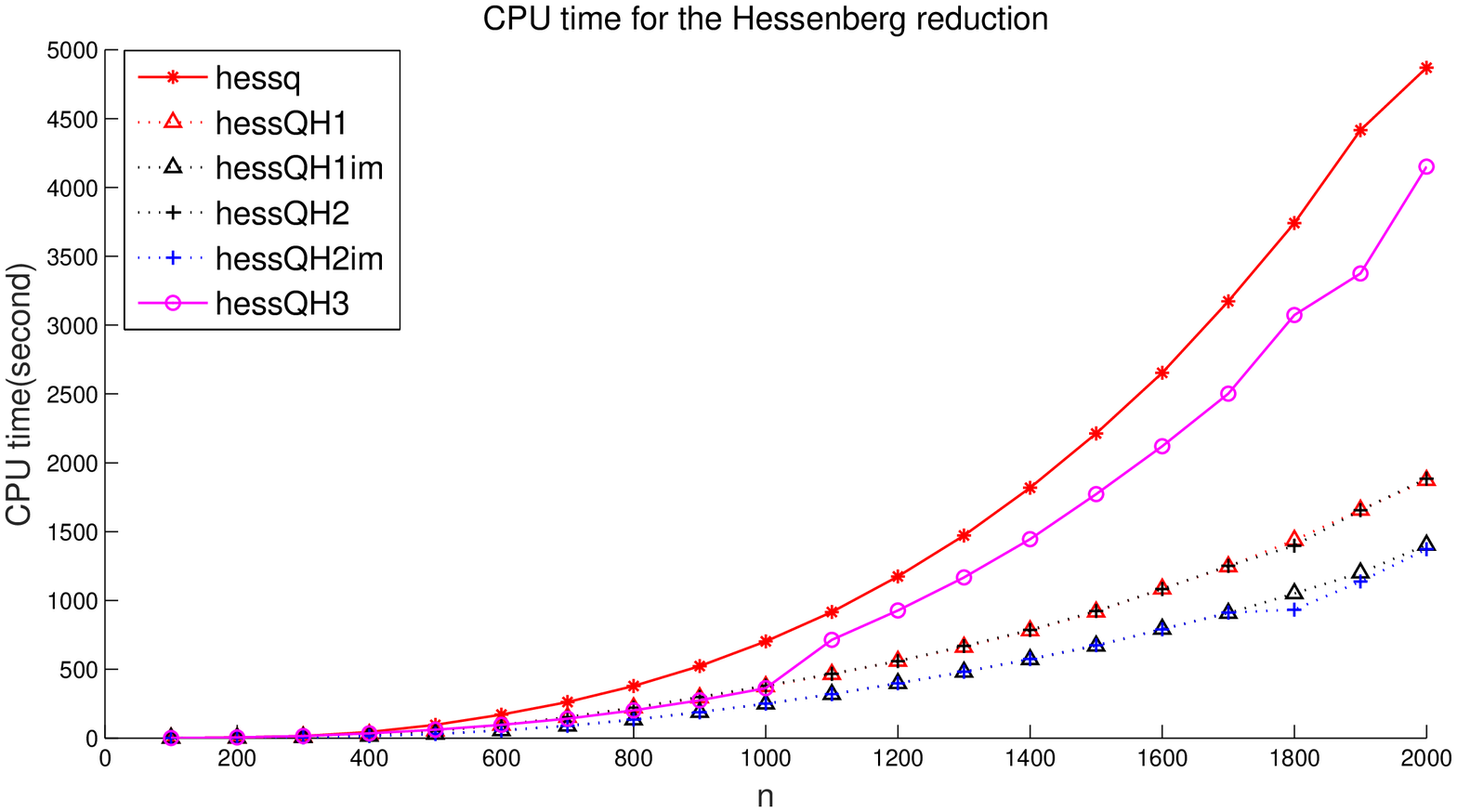}
\includegraphics[width=6cm,height=4.5cm]{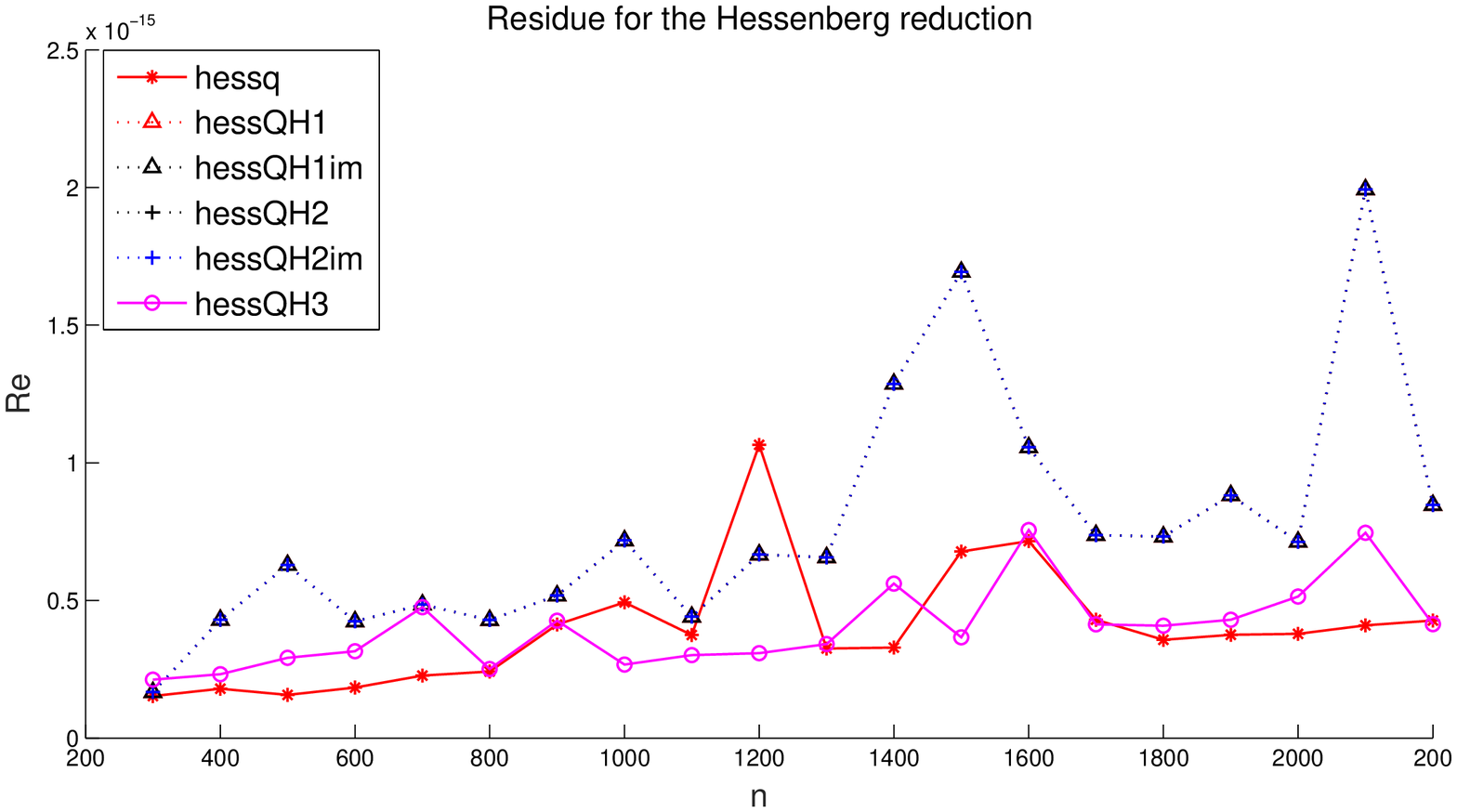}
  \end{center}
  \caption{The CPU times (seconds) and the relative residuals for quaternion Hessenberg reduction}\label{f:eig2}
\end{figure}
 Figure \ref{f:hess} indicates that  
 \begin{itemize}
 \item  when the dimension is large,  the real structure-preserving algorithms cost less CPU times than the algorithms based  on quaternion operations;  
 \item Algorithm \ref{a:hessreduction1} and  Algorithm \ref{a:hessreduction2} generally are faster  than   the Hessenberg reduction algorithms based on the Householder-based transformations $\mathscr{H}_1$, $\mathscr{H}_2$ and $\mathscr{H}_4$ in \cite{lwzz16};
 \item  and  the residue of  Algorithm \ref{a:hessreduction} is generally smaller than those  of Algorithm \ref{a:hessreduction1} and  Algorithm \ref{a:hessreduction2}.
% \item 
 \end{itemize}
\end{example}

\begin{example}[\bf QR Decompositions of Quaternion Hessenberg  Matrices]\label{ex:hessQR} %
Suppose that 
$$H=H_0+H_1i+H_2j+H_3k:=[H_0,H_2,H_1,H_3]$$
is a random upper  quaternion Hessenberg matrix with the real counterpart  
 $JRS$-symmetric, where $H_{0,1,2,3}\in\mathbb{R}^{n\times n}$.
For n=100:100:4000, we compare the numerical efficiency of the following two quaternion Givens transformations  on the QR decomposition of $H$:
\begin{itemize}
\item \texttt{FGivensQ}: applying fast Givens transformation in \cite{joap05};
\item \texttt{GGivensQ}:  Algorithm \ref{a:HessQRggivens}. %
\end{itemize}

 In the left figure of  Figure \ref{f:hessQR}, the CPU times costed by two algorithms \texttt{FGivensQ} and \texttt{GGivensQ} are for the calculation of the upper $JRS$-triangular  matrix $R:=[R_0,R_2,R_1,R_3]$ 
  and the $Q$ factor $W:=[W_0,W_2,W_1,W_3]$.
In the right figure of  Figure \ref{f:hessQR},  the relative residual is  defined as
\begin{equation*}\label{e:ERR}
Re=\frac{\|A-WR\|_F}{\|A\|_F}.
\end{equation*}   
\begin{figure}\label{f:hessQR}
  \begin{center}
\includegraphics[width=6cm,height=4.5cm]{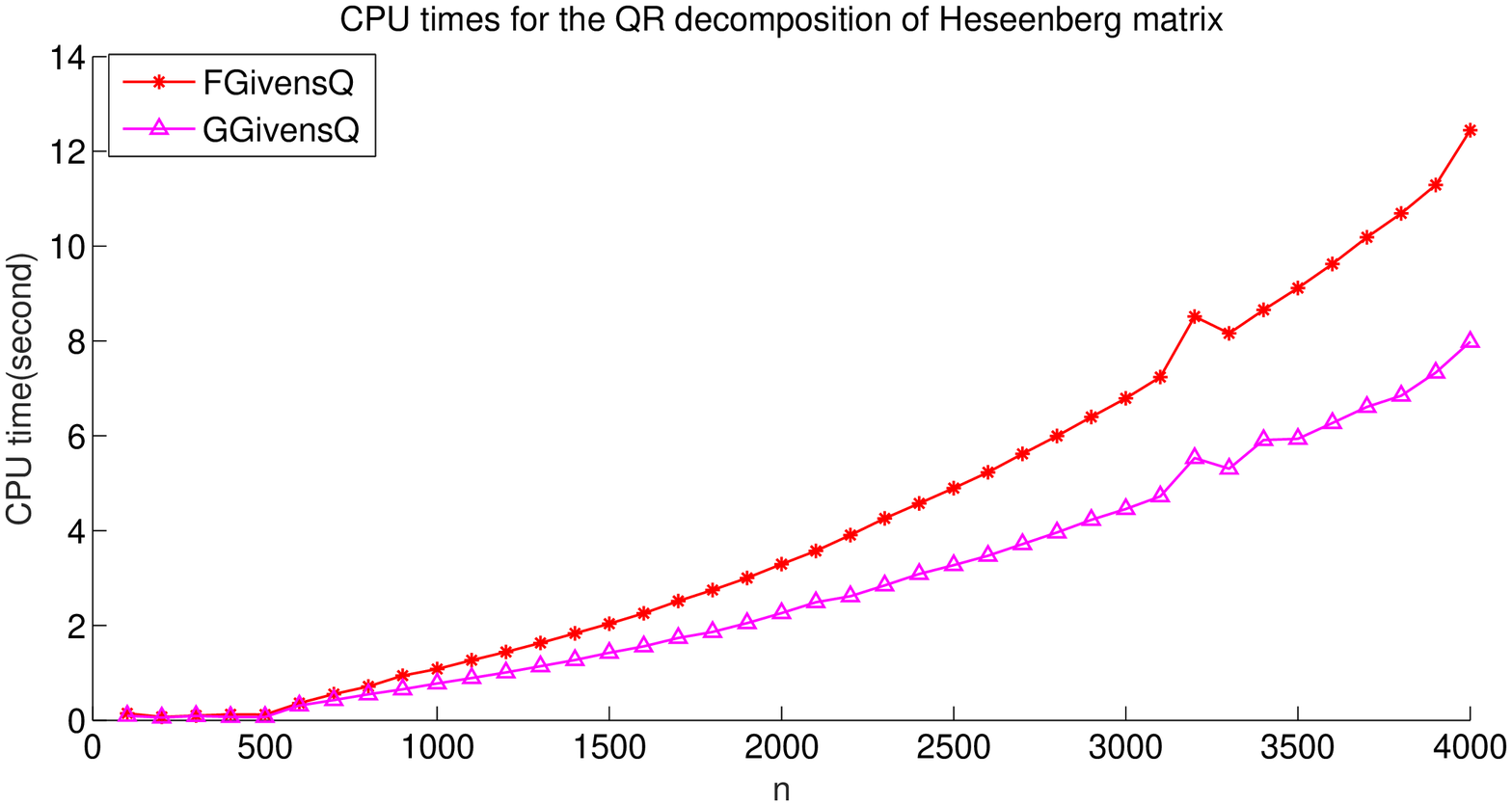}
\includegraphics[width=6cm,height=4.5cm]{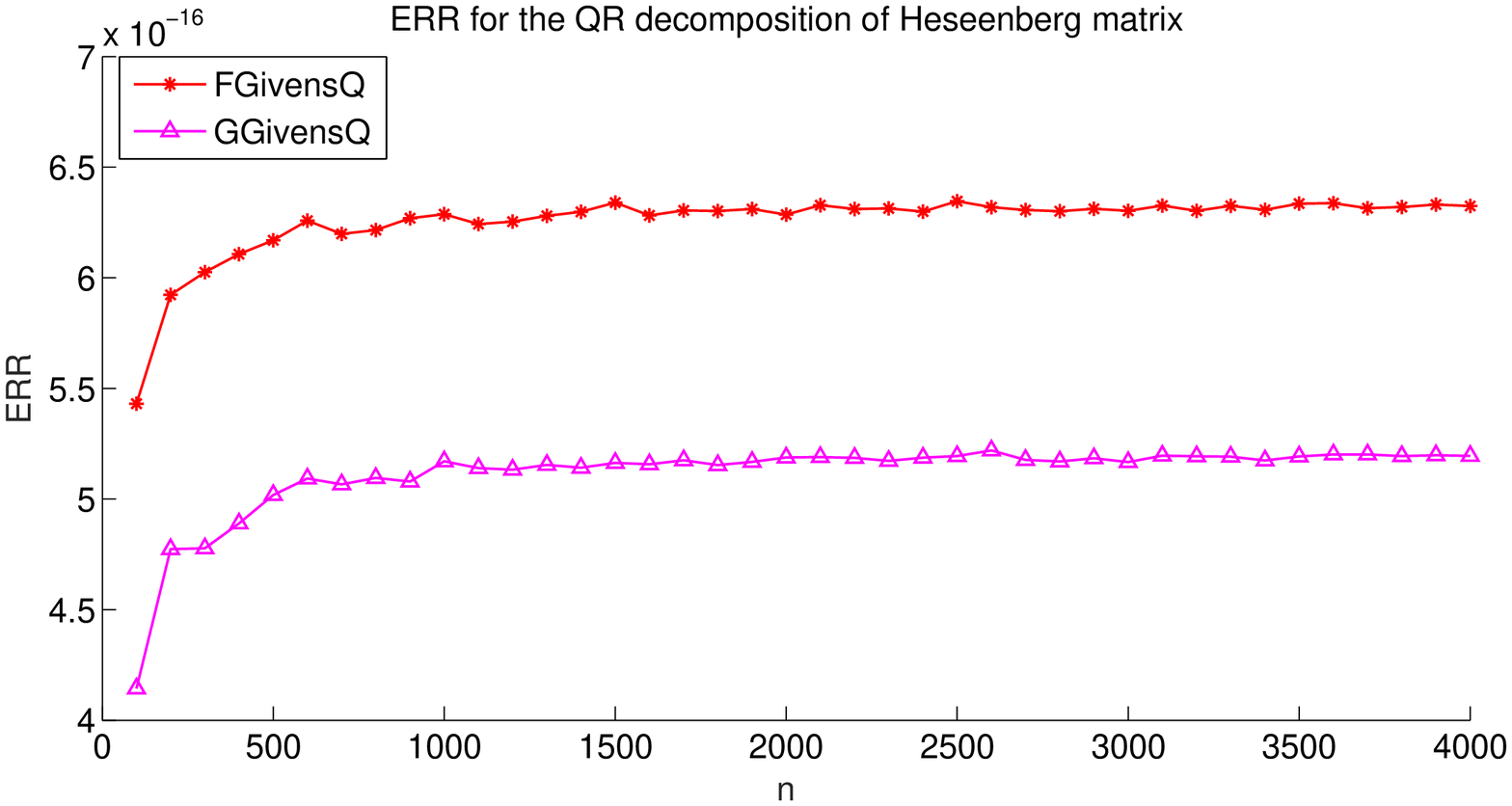}
  \end{center}
  \caption{The CPU times (seconds)  and the relative residuals for quaternion Hessenberg reduction}\label{f:eig2}
\end{figure}

From the numerical results in Figure \ref{f:hessQR}, we can see that  when the dimension is very large  \texttt{GGivensQ} is faster than \texttt{FGivensQ} and the relative residual of \texttt{GGivensQ} is smaller. 
\end{example}
\begin{example} [\bf Hessenberg reduction of $H_F$]\label{ex:PHPhess} 
Suppose that 
$$H_F=H^F_0+H^F_1i+H^F_2j+H^F_3k:=[H^F_0,H^F_2,H^F_1,H^F_3]$$
 is the broken Hessenberg  quaternion matrix in Francis QR step,   where $H^F_{0,1,2,3}$ are $n$-by-$n$ real matrices  as defined in Section \ref{s:doubleim}.
For n=4000:100:7000, we compare the numerical efficiency of  the following  algorithms on Hessenberg reduction of $H^F$:
Algorithm \ref{a:hessreduction1}(\texttt{hessQ1im}), 
Algorithm \ref{a:hessreduction2}(\texttt{hessQH2im}),
Algorithm \ref{a:hessreduction}(\texttt{hessQH3}),
the Hessenberg reduction based on fast Givens transformation (\texttt{hessQ-FGivensQ}),
 and Algorithm \ref{a:HessQRggivens} (\texttt{hessQ-GGivensQ}). 
In the left figure of  Figure \ref{f:PHPhess}, the CPU times costed by four algorithms  are for the calculation of the upper $JRS$-Hessenberg  form $\widehat{H}:=[\widehat{H}_0,\widehat{H}_2,\widehat{H}_1,\widehat{H}_3]$
 %as in \eqref{e:hess} 
  and the orthogonally $JRS$-symplectic matrix $\widehat{W}:=[\widehat{W}_0,\widehat{W}_2,\widehat{W}_1,\widehat{W}_3]$. 
%  as in \eqref{e:w}.
In  the right figures of  Figure \ref{f:PHPhess},  the backward error is  defined as
\begin{equation*}\label{e:residual}
ERR=||H_F\widehat{W}-\widehat{W}\widehat{H}||_F.
\end{equation*}
\begin{figure}\label{f:PHPhess}
  % Requires \usepackage{graphicx}
  \begin{center}
\includegraphics[width=6cm,height=4.5cm]{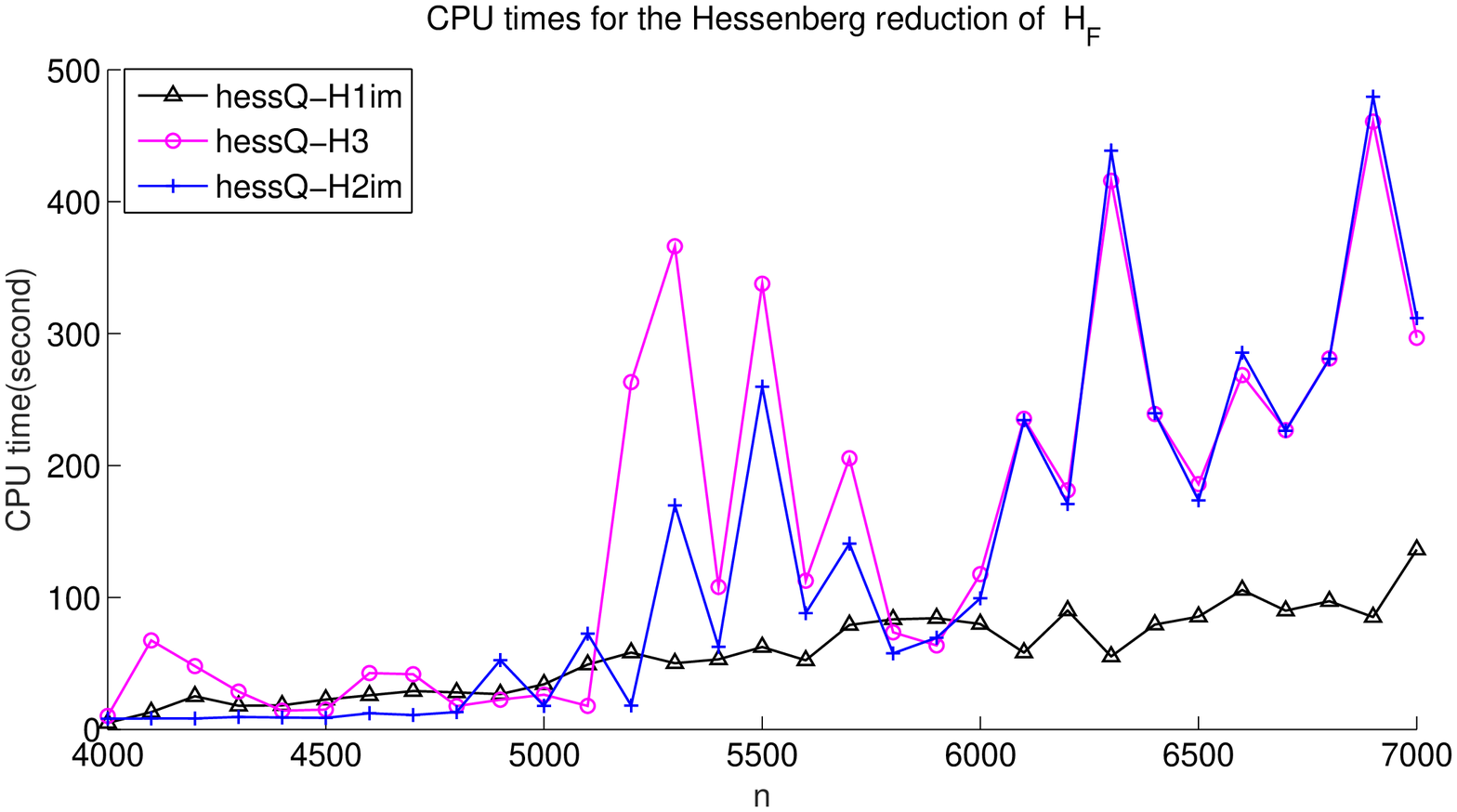}
\includegraphics[width=6cm,height=4.5cm]{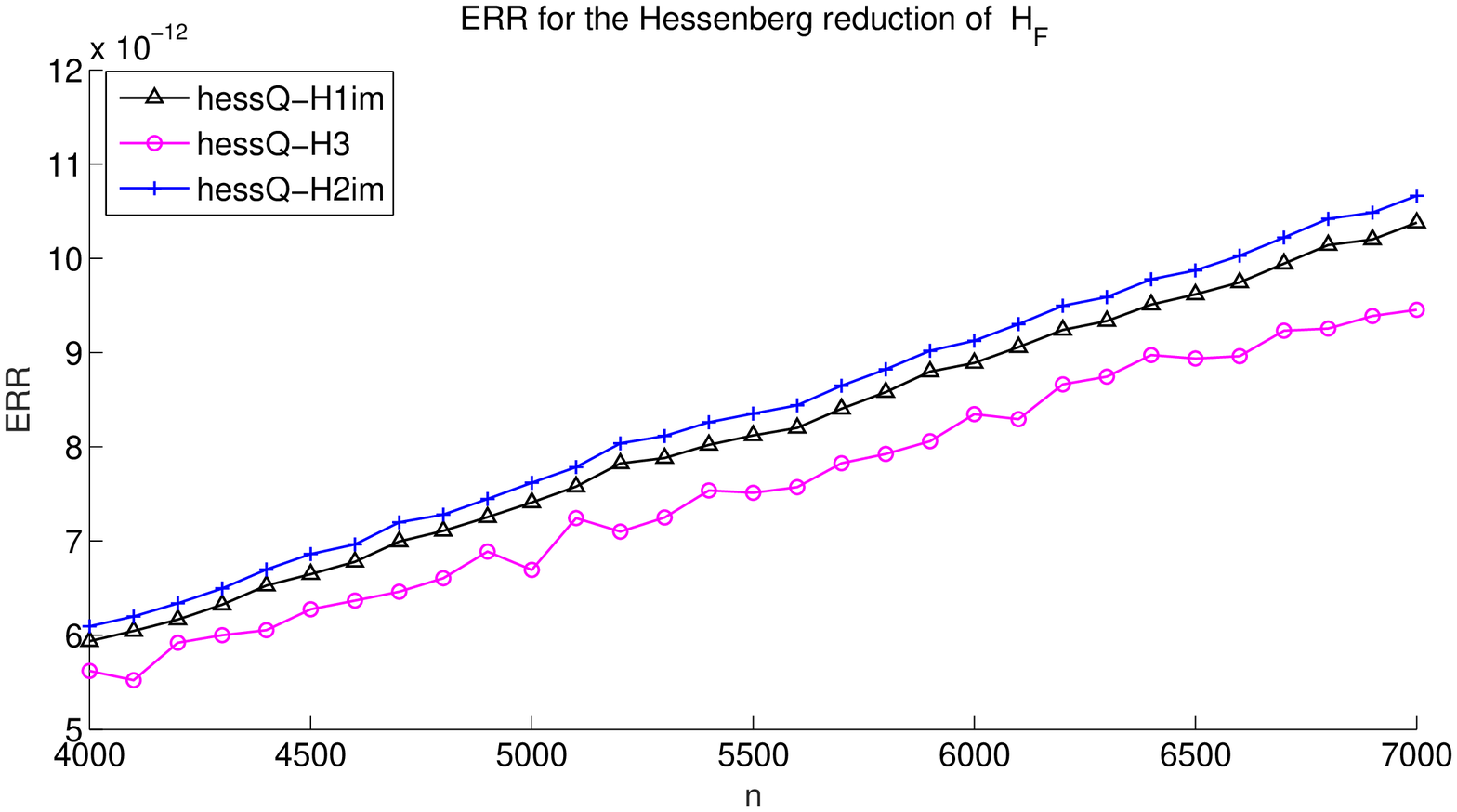}
\includegraphics[width=6cm,height=4.5cm]{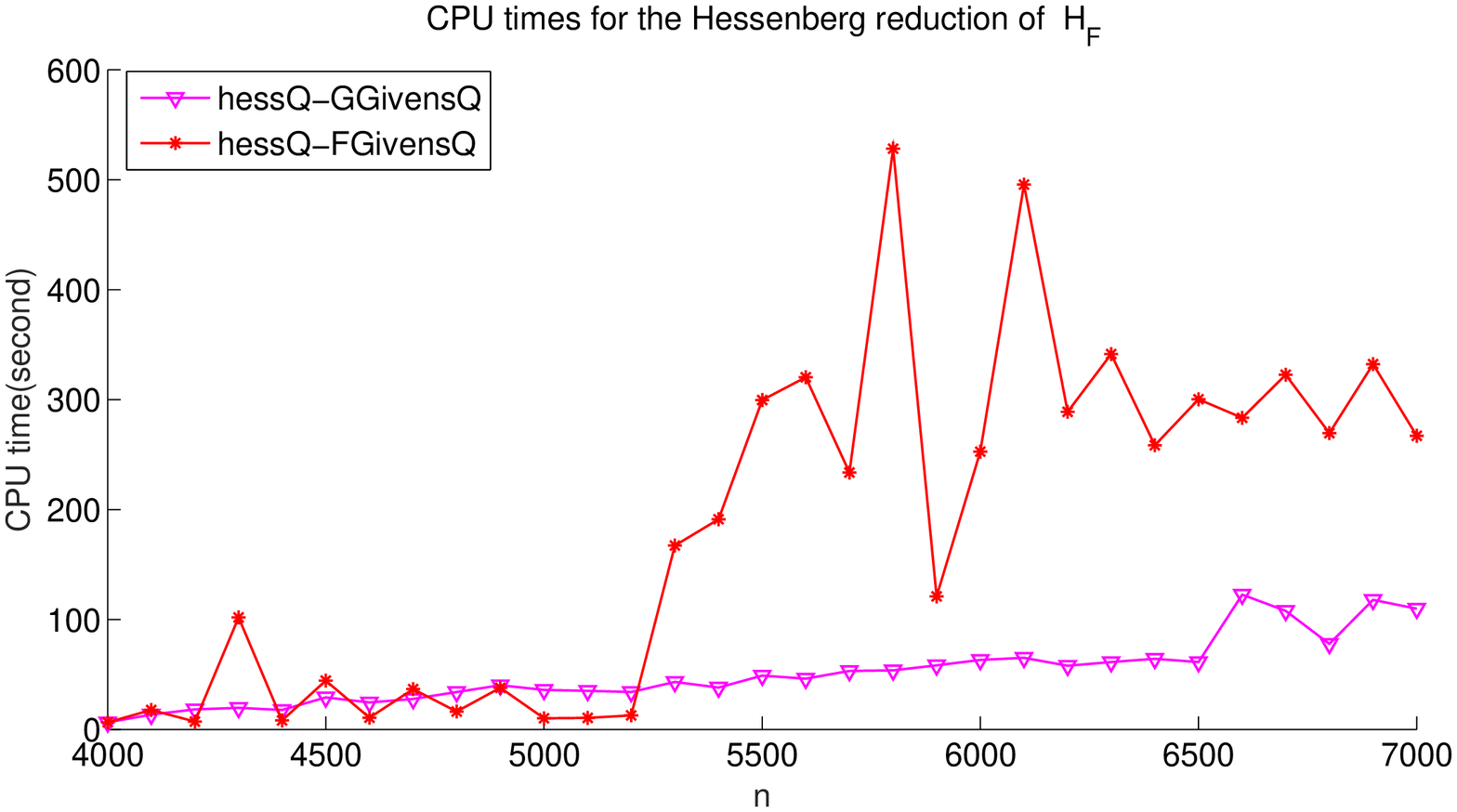}
\includegraphics[width=6cm,height=4.5cm]{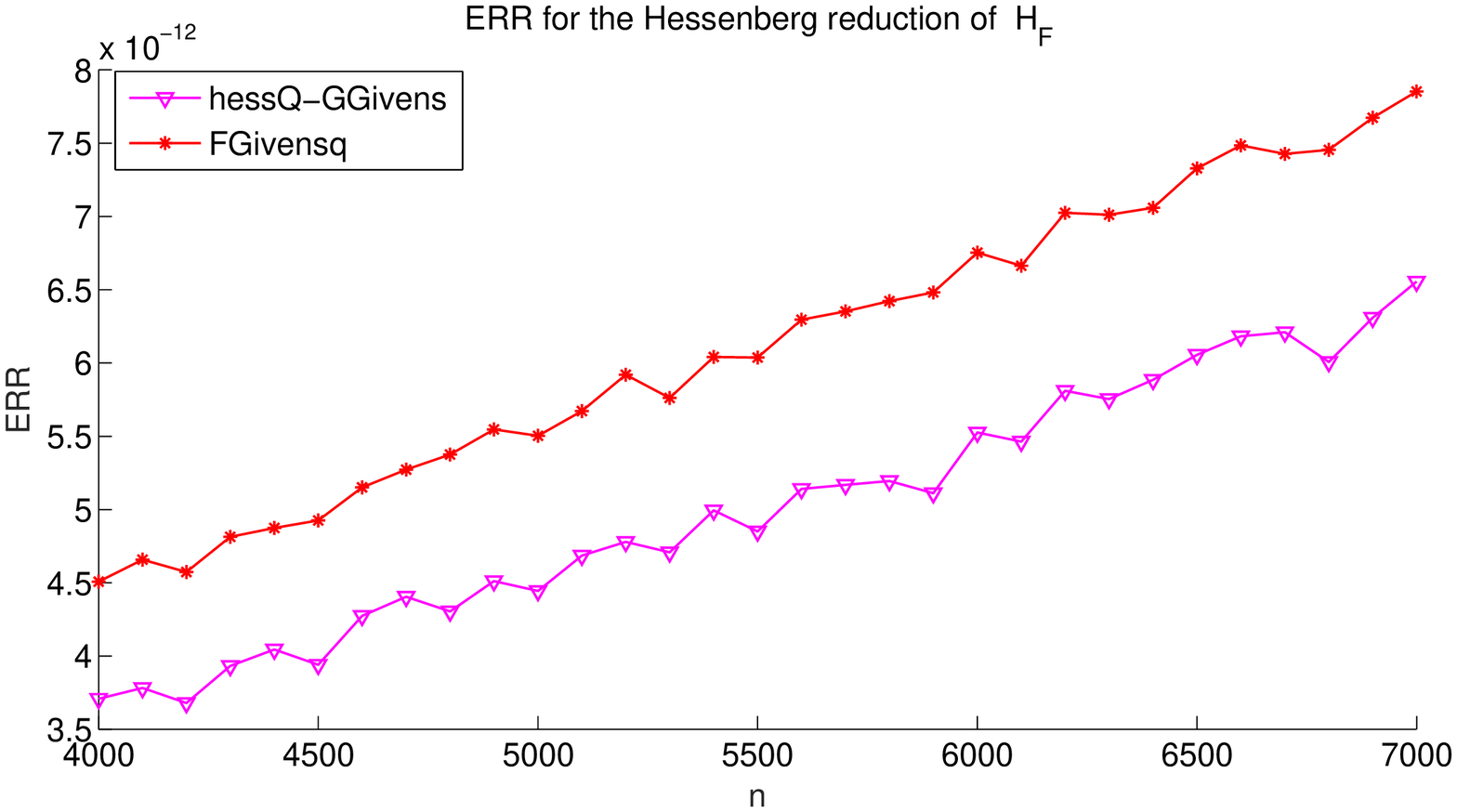}
  \end{center}
  \caption{The CPU times (seconds) and the relative residuals for quaternion Hessenberg reduction}\label{f:eig2}
\end{figure}
\end{example}

\begin{example}[\bf Schur Decompositions of Quaternion Matrices]\label{ex:Schur} 
A newly proposed technique of  the copyright protection of color image is  the blind watermarking scheme based on Schur decomposition.
The features obtained by  Schur decomposition are used for embedding watermark and extracting watermark in the blind manner. 
These watermarking algorithms have  a very good performance, such as in the aspects of the invisibility, robustness, computational complexity, security, capacity etc.; see \cite{such16} for more details.

We apply Algorithm \ref{a:schurform} to compute the quasi upper-triangular Schur decompositions of purely imaginary quaternion matrices denoting color images.
The color image for testing is the standard Lena image of order $512$, denoted by $M=M_1i+M_2j+M_3k:=[0,M_2,M_1,M_3]$, where all elements of $M_{1,2,3}\in\mathbb{R}^{512\times 512}$ are nonnegative but not bigger than $1$.

Let $n$ denote the order of the principle submatrix of $M$. 
For n=12:10:512, we compare the numerical efficiency of two QR algorithms with different kinds of shift:
\begin{itemize}
\item Quaternion QR Algorithm \cite[Algorithm A5]{bbm89}  (QRASq);
\item Algorithm \ref{a:schurform}  with the shift suggested in Section \ref{ss:qralg} (QRASQ).
\end{itemize}

The CPU times reported in Figure \ref{f:schur}  are for the calculation of the $JRS$-Schur  form $T:=[T_0,T_2,T_1,T_3]$ 
 and the orthogonally $JRS$-symplectic matrix  $W$.
\begin{figure}\label{f:schur}
  \begin{center}
  \includegraphics[width=6cm]{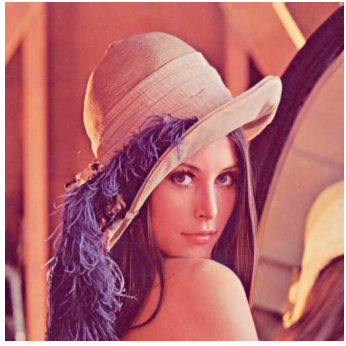}
\includegraphics[width=6cm]{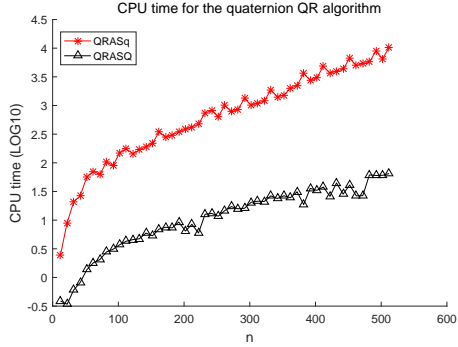}
  \end{center}
  \caption{Lena image and the CPU times (seconds) for Schur decompositions }\label{f:eig2}
\end{figure}
\end{example}

\section{Conclusion}\label{s:conclusion}
A  structure-preserving 
QR algorithm  is presented to calculate the quasi upper-triangular Schur forms of  quaternion matrices. 
The strategy is to preserve the algebraic symmetry of the real counterpart in the processing  and to be in real arithmetic.  
The storage and cost of the newly proposed algorithm are  reduced to 
 the same level of the traditional QR  algorithm in quaternion arithmetic with same  accuracy and stability.
The main contribution of this paper can be concluded as follows.
\begin{itemize}
  \item  
  Prove that once the first column of each block  of the orthogonally $JRS$-symplectic reduction matrix is decided,  the upper $JRS$-Hessenberg form is unique under the  similarity transformation by a diagonal matrix; propose the Francis $JRS$-QR step and a  QR algorithm for computing the real $JRS$-Schur form with preserving the upper $JRS$-Hessenberg structure. 
  \item   Define a novel quaternion Givens transformation and apply it to compute the QR decomposition of quaternion Hessenberg matrix;
  develop a new implicit double shift quaternion QR algorithm which only executes real operations and preserves the structures of quaternion matrices.
  \item The newly  proposes  real structure-preserving quaternion QR algorithm only need to store the real part and three imaginary parts and apply real operations on them directly. We are sure that this is a novel method of computing the right eigenvalues of general quaternion matrices. 
\end{itemize}
Numerical examples show that the newly proposed algorithms are fast and reliable, and that  the larger the dimension of the problem,   the better are they than the state-of-the-art algorithms.\\

\end{document}